\documentclass[reqno]{amsart}
\usepackage{amssymb,amscd,verbatim, amsthm,graphics, color,latexsym,amsmath,multicol,mathtools}
\usepackage{breqn}
\usepackage{fancybox}
\usepackage{tikz-cd}
\usepackage{longtable}
\usepackage[all]{xy}
\usepackage[top=2.7cm, bottom=2.7cm, left=2.7cm, right=2.7cm]{geometry}
\newtheorem{theorem}{Theorem}[section]

\newtheorem{corollary}[theorem]{Corollary}
\newtheorem{lemma}[theorem]{Lemma}
\newtheorem{proposition}[theorem]{Proposition}
\newtheorem{definition}[theorem]{Definition}
\newtheorem{remark}[theorem]{Remark}
\newtheorem{example}[theorem]{Example}
\newtheorem{thmintro}{Theorem}

\newcommand{\lpi}{\langle}
\newcommand{\rpi}{\rangle}

\newcommand{\fg}{\mathfrak{g}}
\newcommand{\fh}{\mathfrak{h}}

\newcommand{\fs}{\mathfrak{s}}

\newcommand{\fM}{\mathfrak{M}}

\newcommand{\fX}{\mathfrak{X}}

\newcommand{\bbc}{\mathbb{C}}

\newcommand{\bbt}{\mathbb{T}}

\newcommand{\bbz}{\mathbb{Z}}

\newcommand{\Cc}{\mathcal{C}}
\newcommand{\Cd}{\mathcal{D}}

\newcommand{\Ck}{\mathcal{K}}

\newcommand{\Co}{\mathcal{O}}

\newcommand{\Cs}{\mathcal{S}}

\newcommand{\Cv}{\mathcal{V}}

\newcommand{\Cx}{\mathcal{X}}

\newcommand{\bd}{\mathbf{d}}
\newcommand{\be}{\mathbf{e}}

\newcommand{\bq}{\mathbf{q}}

\newcommand{\bu}{\mathbf{u}}

\newcommand{\cpv}{\begin{center}\begin{minipage}[t]{14cm}\small{\it Proof.} }
\newcommand{\fpv}{\fim\end{minipage}\end{center}}
\newcommand{\fim}{\hfill $\Box$}

\newcommand{\cha}{\mathsf{H}_{t,c}}
\newcommand{\chad}{\mathsf{H}_{t,\check{c}}}
\newcommand{\chazero}{\mathsf{H}_{0,c}}

\newcommand{\reschazero}{\bar{\mathsf{H}}_{0,c}}

\DeclareMathOperator{\sgn}{sgn}

\DeclareMathOperator{\im}{Im}

\DeclareMathOperator{\Hom}{Hom}

\DeclareMathOperator{\End}{End}
\DeclareMathOperator{\Ind}{Ind}

\DeclareMathOperator{\triv}{triv}

\DeclareMathOperator{\Irr}{Irr}
\DeclareMathOperator{\Sym}{Sym}
\DeclareMathOperator{\vol}{vol}
\DeclareMathOperator{\coker}{coker}

\DeclareMathOperator{\Rep}{Rep}

\DeclareMathOperator{\Ext}{Ext}

\DeclareMathOperator{\ghom}{Homgr}
\edef\epi{\expandafter[twoheadrightarrow]}
\edef\mono{\expandafter[hookrightarrow]}
\def\xar{\expandafter\arrow}

\numberwithin{equation}{section}
\makeatletter
\def\@seccntformat#1{%
  \protect\textup{\protect\@secnumfont
    \ifnum\pdfstrcmp{subsection}{#1}=0 \bfseries\fi
    \csname the#1\endcsname
    \protect\@secnumpunct
  }%
}  
\makeatother

\begin{document}

\title{Dirac induction for rational Cherednik algebras}

\author
{Dan Ciubotaru}
\author
{Marcelo De Martino}
        \address{Mathematical Institute, University of Oxford, Oxford OX2 6GG, UK}
        \email{dan.ciubotaru@maths.ox.ac.uk, marcelo.demartino@maths.ox.ac.uk}

\begin{abstract}
We introduce the local and global indices of Dirac operators for the rational Cherednik algebra $\cha(G,\fh)$, where $G$ is a complex reflection group acting on a finite-dimensional vector space $\fh$. We investigate precise relations between the (local) Dirac index of a simple module in the category $\Co$ of $\cha(G,\fh)$, the graded $G$-character of the module, the Euler-Poincar\' e pairing, and the composition series polynomials for standard modules. In the global theory, we introduce integral-reflection modules for $\cha(G,\fh)$ constructed from finite-dimensional $G$-modules. We define and compute the index of a Dirac operator on the integral-reflection module and show that the index is, in a sense, independent of the parameter function $c$. The study of the kernel of these global Dirac operators  leads naturally to a notion of dualised generalised Dunkl-Opdam operators.
\end{abstract}

\thanks{This research was supported in part by the EPSRC grant EP/N033922/1(2016).}


\maketitle


\section{Introduction}

\subsection{\;\!\!\!\!\!}
Rational Cherednik algebras are a particular example of the symplectic reflection algebras introduced by Etingof and Ginzburg \cite{EG}.  They appear as rational degenerations of the double affine Hecke algebra introduced by I. Cherednik \cite{Ch}. We denote them by $\cha(G,\fh)$ or simply $\cha$, see section \ref{s:RCA}. Their definition depends on a finite complex reflection group $G$ acting on a finite-dimensional complex vector space $\fh$ (and on its dual $\fh^*$), as well as on parameters $t\in\bbc$ and a $G$-conjugation invariant function $c$ on the set of reflections of $G$. As a $\bbc$-module, $\cha$ has a Poincar\'e-Birkhoff-Witt decomposition $\cha = \bbc[\fh]\otimes \bbc G\otimes \bbc[\fh^*]$ and the analogy in structure with that of a universal enveloping algebra of a complex semisimple Lie algebra has led to the theory of category $\Co$ for $\cha$ \cite{GGOR}, denoted here as $\Co_{t,c}(G,\fh)$. This is the category of  left $\cha$-modules that are finitely generated and $\fh$-locally nilpotent (see Definition \ref{d:catO}). This category has standard modules, denoted $M_{t,c}(\tau)$, where $\tau$ is a (finite-dimensional) simple $\bbc G$-module, and the simple modules, $L_{t,c}(\tau)$, are the unique irreducible quotients of $M_{t,c}(\tau)$ \cite{GGOR,Go}. The standard modules have an easy description: as a vector space over $\bbc$,   $M_{t,c}(\tau) = \bbc[\fh]\otimes \tau$. The composition series and the behaviour of the simple modules, on the other hand, is complicated and sensitive to the parameters. For example, when $t=1$, this ranges from $L_{t,c}(\tau) = M_{t,c}(\tau)$ when the parameter $c$ is \emph{regular} (see \cite[Definition 2.15]{DO}) to cases in which $L_{t,c}(\tau)$ is finite dimensional, which is always the case if $t=0$ (see, {\it e.g.}, \cite{BEG}, \cite{Et}, \cite{BP}). A basic problem is to compute  the multiplicity of $L_{t,c}(\tau)$ in a Jordan-H{\"o}lder series of $M_{t,c}(\sigma)$, or equivalently, to compute the graded $G$-character of $L_{t,c}(\tau)$.

In this paper, we introduce what we call a graded \emph{local} theory  and a \emph{global} theory of Dirac operators for rational Cherednik algebras. Here the ``local/global'' names are motivated by the analogy with the setting of real reductive groups, where the local picture would be the Dirac cohomology theory for Harish-Chandra modules \cite{HP,Vo}, while the global setting refers to the Dirac operators acting on spaces of sections of spinor bundles over the real symmetric space \cite{AS,Pa}. Our present work builds on the results of \cite{Ci}, where ungraded local Dirac operators were studied for symplectic reflection algebras, and it is motivated by the Dirac operator results obtained for graded affine Hecke algebras in \cite{BCT,CT,COT}.

In the present setting, the Dirac operators are defined with respect to the Clifford algebra $\Cc=\Cc(V)$, associated to the vector space $V=\fh\oplus\fh^*$ (and the bilinear pairing given by extending the natural bilinear pairing $\fh^*\times\fh\to\bbc$) and its irreducible spin module $S$ realised on the exterior algebra $\bigwedge \fh$. For the class of graded modules  of $\Co_{t,c}(G,\fh)$ with {\it infinitesimal character} (see Definition \ref{d:infchar}), we  define and study ($G$-invariant) graded Dirac operators and the Dirac cohomology (introduced in the ungraded setting in \cite{Ci}), and  compute the {\it index}, $I_D(X)$, of such a module $X$ with infinitesimal character. We also introduce the notion of \emph{Dirac index polynomials}, $d_{\tau,\sigma}(\bq)$ (see (\ref{e:DiracIndPoly}), below), given as the graded multiplicity of the irreducible $\bbc G$-module $\sigma$ in the local index $I_D(L_{t,c}(\tau))$.The problem of computing the graded $G$-character can be recast in terms of Dirac theory as follows: 

\begin{thmintro}\label{t:A}
Let $L_{t,c}(\tau)$ be a simple $\cha$-module. Its graded $G$-character equals
\[\textup{ch}(L_{t,c}(\tau),g,\bq)=\frac{\textup{ch}(I_D(L_{t,c}(\tau)),g,\bq)}{\det_{\fh^*}(1-g \bq)},\quad g\in G.\]
Moreover, the matrix of Dirac index polynomials $[d_{\tau,\sigma}(\bq)]$ is inverse to the matrix $[n_{\tau,\sigma}(\bq)]$, of graded multiplicity of $L_{t,c}(\sigma)$ in $M_{t,c}(\tau)$.
\end{thmintro}
We also relate the Dirac index polynomials to the graded Euler-Poincar\'e pairing (see Corollary \ref{c:EPupshot}):
\begin{equation*}
\textup{EPgr}_{\cha}(M_{t,c}(\tau),L_{t,c}(\sigma))=d_{\sigma,\tau}(\bq),
\end{equation*}
and furthermore to a graded elliptic pairing on the space of $G$-representations, Theorem \ref{t:pairs}. Theorem \ref{t:vogan-conj} and Proposition \ref{p:necessary} give necessary conditions on the irreducible $G$-representations that may appear in the numerator of the character formula above. Theorem \ref{t:A} should be regarded as the rational Cherednik algebra analogue of the character formulae proved for graded affine Hecke algebras in \cite{CT}. At the end of section \ref{s:LocalDirac}, we exemplify these formulae in the case of the restricted Cherednik algebra for $G = W(B_2)$ and $G=W(G_2)$.

\subsection{\;\!\!\!\!\!}
The starting point of the global theory is an adaptation of the notion of \emph{integral-reflection} modules from the graded affine Hecke algebras setting \cite[Section 6.3]{COT} (see also \cite{EOS}) to the present setting of the rational Cherednik algebra $\cha$. In section \ref{s:IRR}, we introduce two types of integral-reflection modules, which we denote by $\fM_{t,c}(\sigma)$ and $\fX_{t,c}(\sigma)$, with $\sigma$ ranging over the finite-dimensional modules of $\bbc G$. The first one is a realisation of the costandard modules in $\Co_{t,c}(G,\fh)$ while the second one is crucial to the theory of global Dirac operators.

To define them we adapt to the case of a complex reflection group (see section \ref{s:IRR}) the ideas from \cite{Gu}. In particular, we introduce and describe the main properties of certain operators in the space of polynomial functions on $\fh$ (an on $\fh^*)$ associated to any reflection of the complex reflection group $G$. These operators, called \emph{integral operators}, and denoted by $I_s$ or $I_s^\vee$ (depending if it acts on $\bbc[\fh]$ or $\bbc[\fh^*]$, respectively), with $s$ a reflection of $G$, are dual to the divided difference operators (also known as BGG operators), under the natural polynomial pairing (see (\ref{e:polyduality})) between $\bbc[\fh]$ and $\bbc[\fh^*]$ (see Proposition \ref{p:calcduality}). We then use the polynomial pairing and the operators above mentioned, to define representations of $\cha$ which are dual to the standard modules in the category $\Co$ for the ``dual'' Cherednik algebra $\chad\cong\cha^{\textup{op}}$. This procedure is inspired by the ``integral-reflection'' representations studied in \cite{EOS,COT} (they also appear in \cite{HO}, without the name integral-reflection).  Unlike $\fM_{t,c}$,  the graded $\cha$-module $\fX_{t,c}(\tau)$ is not in category $\Co$, but it is shown in Proposition \ref{p:filtration} that there is an ascending filtration $F_0(\fX_{t,c}(\tau))\subseteq F_1(\fX_{t,c}(\tau))\subseteq \cdots$ whose union $\cup_n F_n(\fX_{t,c}(\tau))$ equals $\fX_{t,c}(\tau)$ and each filtered piece $F_n(\fX_{t,c}(\tau))$ is in category $\Co$ (these filtered pieces can also be shown to be dual versions of the  more general standard modules $\Delta_n(\tau)$ described in \cite[Section 2.3.3]{GGOR}). A feature that both families of integral-reflection modules $\fM_{t,c}(\tau)$ and $\fX_{t,c}(\tau)$ have in common is that the action of $\fh$ has a very simple description by $\fh$-directional derivatives as opposed to the conventional left action of $\fh$ on $\cha$ by means of deformed derivations in the direction of $y\in \fh$ ({\it i.e.,} by Dunkl operators). This  yields an easy description of the $\fh$-singular vectors (see Definitions \ref{d:M-singvec} and \ref{d:X-singvec}) in both cases.

 A global Dirac operator, see section \ref{s:GlobalDirac}, depends on an irreducible module $\tau$ of $\bbc G$ and it is defined as a linear map 
\begin{equation*}
D_\tau:\fX_{t,c}(\tau\otimes S)\to\fX_{t,c}(\tau\otimes S)
\end{equation*}
 that commutes with the integral-reflection $\cha$-action constructed. We prove:
\begin{thmintro}\label{t:B}
When $t=1$, the kernel and the cokernel of $D_\tau$ are in $\Co_{t,c}(G,\fh)$. 
\end{thmintro}
In other words, the global Dirac operator satisfies a ``Fredholm'' condition: even though $D_\tau$ is an endomorphism of $\fX_{t,c}(\tau\otimes S)$ and this space is not finitely generated over $\cha$, the kernel and cokernel of $D_\tau$ are finitely generated (Theorem \ref{t:fredholm} and Corollary \ref{c:fredholm}). An important step in the proof of Theorem \ref{t:B} is a precise formula for the square $D_\tau^2$ (Proposition \ref{p:D-tau-squared}).

When $t=0$, we may work with the restricted rational Cherednik algebra $\bar{\mathsf H}_{0,c}$ \cite{Go}. In that case, all the standard and simple modules are finite dimensional, and so is the restricted integral-reflection $\bar{\mathsf H}_{0,c}$-module. Hence, automatically, the kernel and cokernel of the restricted operator $D_\tau$ are finite dimensional and in the restricted category $\Co$.

\subsection{\;\!\!\!\!\!}
For both $t=0$ and $t=1$, the module $\fX_{t,c}(\tau\otimes S)$, as well as the kernel and cokernel of $D_\tau$, are objects in the category of locally $\fh$-nilpotent $\cha$-modules. We define  the \emph{global index}, $I(\tau)$, as their difference in the Grothendieck group of this category  (actually not for $D_\tau$, but for its even part $D_\tau^+$) and we then prove:
\begin{thmintro}
The global index of $\tau$ is $I(\tau) = \fM_{t,c}(\tau)$, for all the parameters $t,c$.
\end{thmintro}
This result should be regarded as a manifestation of the rigidity of the Dirac index. A second principle in this theory is that the local Dirac kernel should control the global Dirac kernel, and this provides the bridge between the two settings. More precisely, we have the following result:

\begin{thmintro}[see Proposition \ref{p:local-global}]
Let $Y$ be a module in $\Co_{t,c}(G,\fh)$. Then 
\begin{equation*}
\begin{aligned}
\ghom_{\cha}(Y,\ker D_\tau)&=\ghom_{G}(\tau^*,\ker D_{Y^\dagger}),\\
\ghom_G(\tau^*,\coker D_{Y^\dagger})&\hookrightarrow\ghom_{\cha}(Y, \coker D_\tau),
\end{aligned}
\end{equation*}
where $Y^\dagger$ is the contragredient module of $Y$ and $ D_{Y^\dagger}$ is the local Dirac operator with respect to it.
\end{thmintro}
It is an interesting problem to determine the the kernel (and cokernel) of the global Dirac operators and to provide in this way realisations of the simple $\cha$-modules. The occurrence of simple modules in $\ker D_\tau$ is controlled by the $\fh$-singular vectors in $\ker D_\tau$. As mentioned before, the space of $\fh$-singular vectors of $\fX_{t,c}(\tau\otimes S)$ is easy to determine and a preliminary study of the kernel of $D_\tau$ on that space, when $t=1$ (see subsection \ref{ss:globDiracinSingVecs}), has led us to consider a complex of $G$-modules that depends of the parameter function $c$ and that is dual to the one studied in \cite[Sections 2.2 and 2.3]{DO}. We obtain (see Theorem \ref{t:dunkl-opdam}) a result analogous to the one obtained by Dunkl and Opdam \cite[Theorem 2.9 and Corollary 2.14]{DO}, which may be used to characterise the singular parameters $c$ (see Remark \ref{r:DO}). This indicates that the determination of the nontrivial homology groups of this generalised Dunkl-Opdam complex is closely related to the realisation of the simple $\cha$-modules in the kernel of the Dirac operator. We shall study these questions in future work.

\section{The rational Cherednik algebra}\label{s:RCA}
In this section, we recall the definition of the rational Cherednik algebra \cite{EG}, and of its category $\Co$ \cite{GGOR}; see also \cite[section 3]{EM} for an expository account. 

\subsection{Definitions} Let $\fh$ be a finite-dimensional $\bbc$-vector space of dimension $r$ and $G\subseteq GL(\fh)$ be a finite group generated by reflections.  We shall assume throughout that $G$ acts irreducibly on $\fh$. Thus, for example, if $G = S_n$, we have $r = n-1$ and not $n$. Denote by $\Cs\subseteq G$ the set of reflections. Given $s\in\Cs$ we let $\lambda_s^{-1},\lambda_s\in\bbc$ be the unique nontrivial eigenvalues of $s$ on $\fh$ and $\fh^*$, respectively. We choose eigenvectors $\alpha_s\in \fh^*$ and $\alpha^\vee_s\in \fh$ such that $\alpha_s(\alpha^\vee_s) = 2$. Let $c:\Cs\to\bbc$ be a conjugation invariant function.
\begin{definition}\label{defn:RCA}
The \emph{rational Cherednik algebra} $\cha(G,\fh)$ is the unital, associative algebra over $\bbc$ given as the quotient of $\bbt(\fh\oplus \fh^*)\otimes\bbc G$ modulo the relations
\begin{equation}\label{e:relations}
\begin{aligned}
&[y,y'] = 0 = [x,x'] ,& y,y'\in \fh,x,x'\in \fh^*,\\
&gy = g(y) g, \quad gx = g(x) g,  & y\in \fh,x\in \fh^*,g\in G,\\
&[y,x] = tx(y) - \sum_s c(s)\alpha_s(y)x(\alpha_s^\vee) s,& y\in \fh,x\in \fh^*.
\end{aligned}
\end{equation}
\end{definition}
\begin{remark}
We shall consistently use the shorthand notation $\cha = \cha(G,\fh)$. Given a parameter function $c$, we define $\check{c}:\Cs\to\bbc$ so that $\check{c}(s) = c(s^{-1})$. Since $G$ is naturally a subgroup of $GL(\fh^*)$ we can also consider the algebra $\chad(G,\fh^*)$. We shall use the abbreviated notation $\chad$ for $\chad(G,\fh^*)$. It will be important to us to consider certain naturally defined dualities between $\cha$ and $\chad$. In formulae, we have $\lambda_s\mapsto \lambda_s^{-1}$ if we exchange $\cha$ and $\chad$. We shall always write $y, y'\in\fh, x,x'\in\fh^*,p,p'\in\bbc[\fh]$ and $q,q'\in\bbc[\fh^*]$. Furthermore, if $\mu\neq 0$, we have $\cha\cong \mathsf{H}_{\mu t,\mu c}$, so the important cases are $t=0$ and $t=1$.
\end{remark}

\subsection{Formulae}
As a vector space, \cite{EG} we have $\cha = \bbc[\fh]\otimes\bbc[\fh^*]\otimes\bbc G$.  It is useful to have a formula for the last relation in (\ref{e:relations}) for a polynomial $p\in\bbc[\fh] = \Sym \fh^*$, in the place of $x\in \fh^*$. 

\begin{lemma}\label{l:reflections}
With the choices of eigenvectors made, we have, for all $x\in \fh^*$ and $y\in \fh$, that
\[s(x) = x - d_sx(\alpha_s^\vee)\alpha_s,\quad s(y) = y - d^\vee_s\alpha_s(y)\alpha^\vee_s,\]
where $d_s= \frac{(1-\lambda_s)}{2}$ and $d^\vee_s=\frac{(1-\lambda^{-1}_s)}{2}$. \end{lemma}

\begin{proof}
Clearly $s$ fixes every $x\in(\alpha_s^\vee)^\perp$, and also $s(\alpha_s) = \alpha_s - (1-\lambda_s)\alpha_s = \lambda_s\alpha_s.$ Similar for $s(y)$.
\end{proof}

\begin{remark}
For the formulae for $s^{-1}$, we use that $\lambda_{s^{-1}} = \lambda^{-1}_s$.
When $G$ is a real reflection group, we have $d_s = 1 = d_s^\vee$.
\end{remark}

\noindent Given $s\in \Cs$, let $\Delta_s$ denote the divided difference operator
\[\Delta_s(p) = \frac{p-s(p)}{\alpha_s}.\]
Actually, we have $\Delta_s = \Delta(s,\alpha_s)$, but we shall typically omit the dependence on the choice of $\alpha_s$. One checks that $p - s(p) \in \alpha_s\bbc[\fh]$, so that this is a well-defined map $\bbc[\fh]\to \bbc[\fh]$. We shall write $\Delta_s^\vee:\bbc[\fh^*]\to \bbc[\fh^*]$ for $\Delta(s,\alpha_s^\vee)$.

\begin{proposition}\label{p:commutation}
 For $y\in \fh$ and $p\in\bbc[\fh]$, we have, where $d_s= \frac{(1-\lambda_s)}{2}$, that
\[[y,p] = t\partial_y(p) - \sum_s\frac{c(s)}{d_s}\alpha_s(y)\Delta_s(p)s.\]
\end{proposition}

\begin{proof}
The proof will be by induction on the degree of $p$. We can assume $p$ is a monomial. From Lemma \ref{l:reflections}, it is clear that the last relation of (\ref{e:relations}) can be written as
\[
[y,x] = t\partial_y(x) - \sum_s \frac{c(s)}{d_s}\alpha_s(y)\Delta_s(x) s,
\]
so the formula holds when $\deg p = 1$. Further, we note that for any $p',p\in\bbc[\fh]$, we have $\Delta_s(p'p) = p'\Delta_s(p) + \Delta_s(p')s(p)$. That said, assuming the result is true for a monomial $p$ of degree $d$, we have, for any $x\in \fh^*$, that
\begin{align*}
[y,x p] &= [y,x]p + x[y,p]\\
&= t(\partial_y(x)p + x\partial_y(p)) - \sum_s\frac{c(s)}{d_s}\alpha_s(y)(\Delta_s(x)s(p) + x\Delta_s(p))s,
\end{align*}
and thus the result.
\end{proof}

\begin{proposition}\label{p:commutation2}
For $x\in \fh^*$ and $q\in\bbc[\fh^*]$, we have, where $d^\vee_s= \frac{(1-\lambda^{-1}_s)}{2}$, that
\[[q,x] = t\partial_x(q) - \sum_s\frac{c(s)}{d^\vee_{s}}x(\alpha^\vee_s)\Delta^\vee_s(q)s.\]
\end{proposition}

\begin{proof}
The proof is similar to the previous proposition.
\end{proof}

\begin{remark}
The formulae of Propositions \ref{p:commutation} and \ref{p:commutation2} do not depend on the choice of $\alpha_s$ and $\alpha_s^\vee$ made. Similar equations hold, {\it mutatis mutandis}, for $\chad = \bbc[\fh^*]\otimes\bbc[\fh]\otimes\bbc G$.
\end{remark}

\subsection{Grading} One regards $\cha$ as a $\mathbb Z$-graded algebra by giving $x\in\fh^*$ degree $1$, $y\in\fh$ degree $-1$, and $g\in G$ degree $0$. Define the \emph{Euler element} 
\begin{equation*}
\be\bu=\sum_i x_iy_i+\frac 12 \dim \fh-\sum_{s\in \Cs} \frac{c(s)}{d_s} s \in \cha,
\end{equation*}
where $\{x_i\}$ and $\{y_i\}$ are dual basis of $\fh^*$ and $\fh$, respectively. When $G$ is a real reflection group, $\be\bu=\frac 12\sum_i(x_iy_i+y_ix_i)$. The element $\be\bu$ is $G$-invariant and has the property that
\[ [\be\bu,x]=tx,\ [\be\bu,y]=-ty,\quad x\in\fh^*,\ y\in\fh.
\]
This means that the grading considered for $\cha$ is inner, given by $\be\bu$. When $t=0$, $\be\bu$ is in fact central in $\cha$. 

Throughout, we shall adhere to the following notation regarding graded modules for a graded $\bbc$-algebra $A$. If $\Gamma$ is an abelian group such that $A$ is $\Gamma$-graded and $X,Y$ are $\Gamma$-graded $A$-modules, then we shall write $\ghom_A^0(X,Y) = \{f\in\Hom_A(X,Y)\mid f(X_\gamma)\subseteq Y_\gamma\textup{ for all }\gamma\in\Gamma\}$ and $\ghom_A^\gamma(X,Y) = \ghom_A^0(X[\gamma],Y)$, in which $X[\gamma]$ is the $\Gamma$-graded module equal to $X$ as a vector space and that satisfies $X[\gamma]_\delta = X_{\delta-\gamma}$, for all $\delta\in\Gamma$. Then, $\ghom_A^\Gamma(X,Y)$ (or $\ghom_A^\bullet(X,Y)$, if $\Gamma$ is clear from the context) will be given by
\[\ghom_A^\Gamma(X,Y) = \bigoplus_{\gamma\in\Gamma} \ghom_A^\gamma(X,Y).\]

\subsection{Duality} We recall a natural duality between $\cha$ and $\chad$.
\begin{lemma}\label{l:duality}
The map $\gamma:\chad^{\textup{op}}\to\cha$ given by $\gamma|_\fh = id = \gamma|_{\fh^*}$ and $\gamma(g) = g^{-1}$, extends to an isomorphism between $\chad^{\textup{op}}$ and $\cha$.
\end{lemma}

\begin{proof}
We have to check that the map $\gamma$ preserves the relations in (\ref{e:relations}). We have, for example,
\[\gamma(gy) = yg^{-1} = g^{-1}g(y) = \gamma(g(y)g).\]
Also,

\[
\gamma([y,x]^{\textup{op}})= \gamma([x,y])=t x(y) - \sum_s \check{c}(s)x(\alpha_s^\vee)\alpha_s(y) s^{-1},
\]
and thus the relation follows, as $\alpha_{s^{-1}} = \mu\alpha_s$ and $\alpha_{s^{-1}}^\vee = \mu^{-1}\alpha_s^\vee$ for a nonzero $\mu\in\bbc$.
\end{proof}
\noindent Thus, if $X$ is an $\cha$-module, the linear dual $X^*$ is endowed with an $\chad$-action via
\begin{equation*}\label{e:dualaction}
(h(v^*),v) = (v^*,\gamma(h)(v)),
\end{equation*}
for all $h\in \chad$, $v^*\in X^*$ and $v\in X$, where $(\cdot,\cdot)$ denotes the bilinear pairing between $X$ and $X^*$.
This yields an $\chad$-action for any $\chad$-invariant subspace $X^\prime\subseteq X^*$. Denote also by $\gamma$ the  similarly defined homomorphism $\cha^{\textup{op}}\to\chad$. 

\subsection{The centre of $\cha$} The centre of $\cha$ is $\bbc$ when $t\neq 0$ \cite{EG}. When $t=0$, the algebra $\chazero$ has a large centre, in particular \cite{EG,Go}
\[
\bbc[\fh]^G\otimes\bbc[\fh^*]^G\subseteq Z(\chazero).
\]
As in \cite{Go}, consider the dual morphism 
\begin{equation*}
\Upsilon: \Cx_c(G)=\textup{Spec } Z(\chazero)\longrightarrow \fh/G\times \fh^*/G.
\end{equation*}
The space $\Cx_c(G)$ is called the \emph{generalised Calogero-Moser space} \cite{EG}. We are mainly interested in the fibre $\Upsilon^{-1}(0)$ over $\{0\}\times\{0\}\in \fh/G\times \fh^*/G$ considered in \cite{Go}. As in {\it loc. cit.}, define the restricted Cherednik algebra 
\begin{equation*}\label{e:restrCA}
\reschazero=\chazero/\langle\bbc[\fh]^G_+\otimes\bbc[\fh^*]^G_+\rangle,
\end{equation*}
where $\bbc[\fh]^G_+$ denote the positive degree part of $\bbc[\fh]^G$ and similarly for $\bbc[\fh^*]^G_+$. The algebra $\reschazero$ is graded (with the grading inherited from $\chazero$), finite dimensional and it is isomorphic as a vector space with $\bbc[\fh]_G\otimes \bbc[\fh^*]_G\otimes \bbc G$. Here, $\bbc[\fh]_G$ and $ \bbc[\fh^*]_G$ are the spaces of $G$-coinvariants. The simple modules for $\reschazero$ are precisely the simple $\chazero$-modules on which $Z(\chazero)$ acts by an element of $\Upsilon^{-1}(0)$.

\subsection{Category $\Co$}Let $\Rep(G)$ denote the set of finite-dimensional modules of $\bbc G$ and $\Irr(G)\subseteq \Rep(G)$ be a (finite) set of representatives of the equivalence classes of the simple ones. 

\begin{definition}[\cite{GGOR}] \label{d:catO}
Let $\Co_{t,c}(G,\fh)$ denote the full subcategory of left $\cha$-modules that are:
\begin{enumerate}
\item $\bbc[\fh]$-finitely generated, and
\item $\bbc[\fh^*]$-locally nilpotent.
\end{enumerate}
Similarly, one also defines the subcategory $\Co_{t,\check c}(G,\fh^*)$ of $\chad$-modules.
\end{definition}

The \emph{standard modules} in $\Co_{t,c}(G,\fh)$ are defined as follows. If $\tau\in\Rep(G)$ we shall consider the $\cha$-module
\[M_{t,c}(\tau) = \cha\otimes_{\bbc[\fh^*]\rtimes\bbc G}\tau= \Ind_{\bbc G}^{\cha} \tau.\]
where $\fh^*$ acts as $0$ on $\tau$. 
The \emph{costandard module} associated to $\tau$ is defined as (see \cite[2.3]{GGOR}) 
\[M^\vee_{t,c}(\tau) =\ghom^\bullet_{\bbc[\fh]\rtimes\bbc G}(\cha,\tau).\] 
Let $L_{t,c}(\tau)$ denote the unique simple quotient of $M_{t,c}(\tau)$, equivalently, the unique simple submodule of $M^\vee_{t,c}(\tau)$. Then $L_{t,c}(\tau)\not\cong L_{t,c}(\tau')$ if $\tau\not\cong\tau'$ and every simple module in $\Co_{t,c}(G,\fh)$ is isomorphic to an $L_{t,c}(\tau)$ as an ungraded $\cha$-module.

Following the standard notation ({\it e.g.}, \cite[section 3.7]{EM}), denote for every irreducible $G$-representation $\sigma$:
\begin{equation*}
N_c(\sigma)=\sum_s\frac {c(s)}{d_s}s\mid_\sigma,\quad h_c(\sigma)=\frac {tr}2-N_c(\sigma).
\end{equation*}
When $t\neq 0$, we say that $L_{t,c}(\sigma)$ and $L_{t,c}(\tau)$ are in the \emph{same block of $\Co_{t,c}(G,\fh)$} if 
\[N_c(\sigma)-N_c(\tau)\in\mathbb Z.\]
We write $\sigma\sim\tau$ when this is the case.

\medskip

Similarly, we have a category $\bar\Co$ for the restricted algebra $\reschazero$ \cite{Go}. The standard modules are the so-called baby Verma modules
\[\bar M_c(\tau)=\reschazero\otimes_{\bbc[\fh^*]_G\rtimes\bbc G}\tau,\]
with unique simple quotients $\bar L_c(\tau)$, $\tau\in\textup{Irr}~G$. Every simple module is isomorphic to one and only one $\bar L_c(\tau)$. Consider the central character map:
\begin{equation*}
\textup{cc}: \textup{Irr}~G \longrightarrow \Upsilon^{-1}(0),\quad  \tau\mapsto \text{the central character of }L_c(\tau).
\end{equation*}
The fibres of this map are called the Calogero-Moser (CM) families of $\textup{Irr}~G$. We write $\sigma\sim\tau$ if $\sigma$ and $\tau$ are in the same CM family. We say that $L_c(\sigma)$ and $L_c(\tau)$ are in the \emph{same block of $\bar\Co$} if $\sigma\sim\tau$.

\section{The Dirac index: local setting} \label{s:LocalDirac}

\subsection{The Clifford algebra} Set $V = \fh\oplus\fh^*$. We extend the natural bilinear pairing $\fh^*\times\fh\to\bbc$ to a bilinear pairing $(\cdot,\cdot):V\times V\to \bbc$, by declaring $( y,y') = 0 = ( x,x')$, for all $y,y'\in\fh$ and $x,x'\in\fh^*$. Let $\Cc = \Cc(V,(\cdot,\cdot))$ be the Clifford algebra over $\bbc$ associated to $V$ and the bilinear pairing indicated, with the relation
\begin{equation}\label{e:Cliff-relation}
vv' + v'v = -2(v,v').
\end{equation}
Since $V = \fh\oplus\fh^*$ is even dimensional, $\Cc$ has $S = \bigwedge\fh$ as the irreducible spin module  with Clifford action $\fs:\Cc\to \End_\bbc(S)$ given by
\begin{equation}\label{e:S-Cliff}
\fs(y)(\omega) = \mu_y(\omega) = y\wedge \omega,\qquad \fs(x)(\omega) = -2\partial_x(\omega),
\end{equation}
for all $y\in\fh,x\in\fh^*$ and $\omega\in\Cs$. Here, $\partial_x$ is the odd-derivation on $S$ characterised by $\partial_xy = x(y),$ as usual. Now let $S^* = \bigwedge\fh^*$. Since the determinant pairing $\lpi\cdot,\cdot\rpi:S^*\times S\to \bbc$ defined by 
\[\lpi x_{i_1}\wedge\ldots\wedge x_{i_\ell},y_{j_1}\wedge\ldots\wedge y_{j_\ell}\rpi = \det(x_{i_p}(y_{j_q}))\]
and extended bilinearly to $S^*\times S$ is nondegenerate, $S^*$ is identified with the dual of $S$. By means of the transpose map $v^t = -v$ for all $v\in V$ and $(ab)^t = b^ta^t$ for all $a,b\in\bbc$, the formula for the action $\fs^*:\Cc\to\End_\bbc(S^*)$ on $S^*$, dual to (\ref{e:S-Cliff}), becomes
\begin{equation}\label{e:S-Cliffd}
\fs^*(y)(\omega^*) = -\partial_y(\omega^*),\qquad \fs^*(x)(\omega^*) = 2\mu_x(\omega^*),
\end{equation}
for all $\omega^*\in S^*$ and $x\in\fh^*,y\in\fh$.
By restricting the Clifford action to $\Cc^{\textup{even}}$, we get a decomposition $S = S^++S^-$ for $\Cc^{\textup{even}}$-invariant subspaces $S^{\pm}$ of $S$. Explicitly, $S^+ = \bigwedge^{\!\textup{even}}\fh$ and $S^-=\bigwedge^{\!\textup{odd}}\fh$.

Fix a basis $\{y_1,\ldots,y_r\}$ of $\fh^*$ and a dual basis $\{x_1,\ldots,x_r\}$ of $\fh$. We define elements $\Cd_x$ and $\Cd_y$ in $\cha\otimes\Cc$ by
\[\Cd_x= \sum_j x_j\otimes y_j,\qquad \Cd_y= \sum_j y_j\otimes x_j,\]
and we set $\Cd = \Cd_x+\Cd_y$. These elements were studied in \cite[Section 4.6]{Ci}. There, formulae for their square were computed and invariance with respect to $G$ was established. We recall these results next.

\subsection{A formula for $\Cd^2$} Define the following elements of $\Cc$:
\begin{equation}\label{e:tau}
\begin{aligned}
&\tau_s^\vee=d_s^\vee \frac{\alpha^\vee_s \alpha_s}2 +1,\quad \tau_s=d_s \frac{\alpha_s \alpha^\vee_s}2 +1,& (s\in\Cs);\\
&\kappa=\sum_i (x_i y_i+1).
\end{aligned}
\end{equation}
As in \cite[Section 4.5]{Ci}, each one of the maps $s\mapsto \tau_s$ and $s\mapsto \tau_s^\vee$ define embeddings of $\mathbb C G$ into $\Cc$. Let $\rho_G:\mathbb C G\to \cha\otimes \Cc$ be defined by extending 
\[s\mapsto s\otimes \tau_s.
\]
Then, $\Cd$ is $G$-invariant:
\begin{equation*}
\rho_G(g)\cdot \Cd\cdot \rho_G(g^{-1})=\Cd\text{ in }\cha\otimes\Cc,\text{ for all }g\in G.
\end{equation*}
\begin{proposition}[{{\it cf.} \cite[Proposition 4.9]{Ci}}]\label{p:D^2}
$\frac 12 \Cd^2=-\be\bu\otimes 1+ t (1\otimes \frac\kappa 2)-\rho_G(\Omega_G),$ where $\Omega_G=\sum_{s\in\Cs} \frac{c(s)}{d_s} s.$
\end{proposition}

\begin{proof}Since the notation is different than in {\it loc. cit.}, we present the short calculation. Firstly, 
\[\Cd_x^2=\sum_{i,j} [x_i,x_j]\otimes y_i y_j=0,
\]
where we used that $y_i y_j=-y_j y_i$ in $\Cc$. Similarly, $\Cd_y^2=0$. Hence
\begin{align*}
\frac 12\Cd^2&=\frac 12\Cd_x\Cd_y+\frac 12\Cd_y\Cd_x\\
&=-\sum_i x_iy_i\otimes 1+\frac 12\sum_{i,j}[y_j,x_i]\otimes x_j y_i \quad \text{(using $y_i x_i=-x_i y_i-2$)}\\
&=-\sum_i x_iy_i\otimes 1+\frac t2 \sum_i 1\otimes x_iy_i-\frac 12\sum_s c(s) s\otimes \sum_{i,j}\alpha_s(y_j) x_i(\alpha_s^\vee) x_j y_i\\
&=-\sum_i x_iy_i\otimes 1+\frac t2 \sum_i 1\otimes x_iy_i-\frac 12\sum_s c(s) s\otimes \alpha_s\alpha_s^\vee\\
&=-\be\bu\otimes 1+t\left(1\otimes \frac\kappa 2\right)-\sum_s c(s) s\otimes \left(\frac{\alpha_s\alpha_s^\vee}2+\frac 1{d_s}\right)\\
&=-\be\bu\otimes 1+t\left(1\otimes \frac\kappa 2\right)-\sum_s c(s) s\otimes \tau_s,
\end{align*}
as required.
\end{proof}

Notice that $\Omega_G$ is $G$-invariant, hence it acts by a scalar  on every irreducible $G$-module $\sigma$. It is also important to recall that in the spin module $S=\bigwedge\fh$, 
\begin{equation*}
\begin{aligned}
\frac\kappa 2&\text{ acts by scalar multiplication by } \left(-\frac r2 +\ell\right)\text{ on }{\bigwedge}^{\!\ell}\fh,\\
\tau_s^\vee&\text{ acts by the reflection }s,
\end{aligned}
\end{equation*}
while in the realisation $S^*=\bigwedge\fh^*$, 
\begin{equation*}
\begin{aligned}
\frac\kappa 2&\text{ acts by scalar multiplication by } \left(\frac r2 -\ell\right)\text{ on }{\bigwedge}^{\!\ell}\fh^*,\\
\tau_s&\text{ acts by the reflection }s.
\end{aligned}
\end{equation*}

\subsection{The local Dirac operator} We recall the definition of Dirac cohomology \cite{Ci}. Let $X$ be a module in $\Co_{t,c}(G,\fh)$. The action of the Dirac element $\Cd$ defines a \emph{local Dirac operator}
\begin{equation*}
D_X:X\otimes S^*\to X\otimes S^*.
\end{equation*}
\begin{definition}[\cite{Ci}]
The \emph{Dirac cohomology} of $X$ is $H_{D}(X)=\ker D_X/\ker D_X\cap \im D_X$. 
\end{definition}

For every $X$ in $\Co_{t,c}(G,\fh)$, $H_D(X)$ is a finite-dimensional $G$-module. We may refine this notion by noting that the action of $\Cd$ maps $X\otimes S^{*,+}$ to $X\otimes S^{*,-}$. Denote the corresponding operators and Dirac cohomologies by
\begin{equation*}
D_X^{\pm}:X\otimes S^{*,\pm}\to X\otimes S^{*,\mp},\quad H_{D}^\pm(X)=\ker D_X^\pm/\ker D_X^\pm\cap \im D_X^\mp.
\end{equation*}
\begin{definition}
The \emph{local Dirac index} of $X$ is $I_D(X)=H_{D}^+(X)-H_D^-(X)$, a virtual $G$-module. 
\end{definition}
We may regard these definitions in the graded setting as follows. In the Clifford algebra $\Cc(V)$, assign degree $1$ to $x\in \fh^*$ and degree $-1$ to $y\in \fh$. Since the defining relations are (see (\ref{e:Cliff-relation}))
\begin{equation*}
xx'=-x'x,\ yy'=-y'y,\ xy+yx=-2(x,y),\quad x,x'\in \fh^*,\ y,y'\in \fh,
\end{equation*}
this assignment defines a $\mathbb Z$-grading on $\Cc(V)$. The standard gradings on $S=\bigwedge \fh$ (with degree $-1$ on $\fh$) and $S^*=\bigwedge \fh^*$ (with degree $1$ on $\fh^*$) make $S,S^*$ into graded $\Cc(V)$-module.

We regard $\cha\otimes \Cc(V)$ as a $\mathbb Z$-graded algebra (not bi-graded) with respect to the grading given by the total degree. Thus $\Cd$ is an element of degree $0$ in $\cha\otimes \Cc(V)$. This implies that $D_X$ is a graded operator for every graded module $X$ in $\Co_{t,c}(G,\fh)$, and then $H_D(X)$ (respectively, $I_D(X)$) is naturally a graded $G$-module (respectively, virtual $G$-module).

\begin{definition}\label{d:infchar}Let $X$ be a graded module in $\Co_{t,c}(G,\fh)$. We say that $X$ has an \emph{infinitesimal character} if:
\begin{enumerate}
\item when $t\neq 0$, $\be\bu$ acts on the $d$-th degree subspace $X_d$ of $X$ by scalar multiplication by $td+k_X$, where $k_X$ is a constant independent of $d$;
\item when $t=0$, the centre of $\cha$ acts by scalars $\chi_X$ on $X$.
\end{enumerate}
\end{definition}

\begin{remark}
A graded module in category $\Co$ is of finite type \cite{GGOR}, i.e., it is degree-wise of finite dimension.
\end{remark}

\begin{lemma}\label{l:D^2-sigma}
Suppose that $X$ is a graded $\cha$-module with infinitesimal character and $\sigma$ an irreducible $G$-representation. If $\sigma$ appears in $X_d\otimes {\bigwedge}^{\!\ell}\fh^*$, then
\[\frac 12 D_X^2\mid_\sigma=-t(d+\ell)-k_X+h_{c}(\sigma).
\]
\end{lemma}

\begin{proof}
This is immediate from Proposition \ref{p:D^2} since $-\be\bu$ acts by $-td-k_X$ on this copy of $\sigma$ and $\kappa/2$ acts by $t(r/2-\ell).$ 
\end{proof}

\begin{proposition}\label{p:local-index}
For every graded module $X$ which has an infinitesimal character, \[I_D(X)=X\otimes S^{*,+}-X\otimes S^{*,-},\] as virtual graded $G$-modules.
\end{proposition}
\begin{proof}
The proof is analogous to that from the setting of graded affine Hecke algebras \cite[Lemma 4.1]{COT} but we need a little more care because the modules are infinite dimensional. Let $n\in \mathbb Z$ be a degree and consider $D_{X,n}^\pm$, the restriction of the operators $D_{X,n}^\pm$ to the degree $n$ components $(X\otimes S^{*,\pm})_n$ of $X\otimes S^{*,\pm}$. Notice that $\dim (X\otimes S^{*,\pm})_n<\infty$ so that we are now in a finite-dimensional setting. Let $\sigma$ be an irreducible representation of $G$ and let $(X\otimes S^{*,\pm})_{n,\sigma}$ denote the isotypic component of $\sigma$ in $(X\otimes S^{*,\pm})_n$. 
Write $(X\otimes S^{*,\pm})_{n,\sigma}=\sum_{\ell=0}^r (X_{n-\ell}\otimes{\bigwedge}^{\!\ell}\fh^*)_\sigma.$  If $v\in (X_{n-\ell}\otimes{\bigwedge}^{\!\ell}\fh^*)_\sigma$, by the assumption that $X$ has an infinitesimal character, we see that:

\begin{enumerate}
\item $\frac 12\Cd^2(v)=-t n-k_X+h_c(\sigma)$, when $t\neq 0$;
\item $\frac 12\Cd^2(v)=-\chi_X(\be\bu)-N_c(\sigma)$, when $t=0$.
\end{enumerate}
The point is that in both situations, $\Cd^2(v)$ depends at most on $n$, hence $\Cd^2(v)=0$ for some $v\in (X\otimes S^{*,\pm})_{n,\sigma}$ {\it if and only if} $\Cd^2=0$ on all $(X\otimes S^{*,\pm})_{n,\sigma}$.
Therefore, either $\Cd_X^-\circ \Cd_X^+$ is invertible on $(X\otimes S^{*,+})_{n,\sigma}$, in which case there is no contribution from $\sigma$ in degree $n$ to $I_D(X)$, or else $\Cd_X^-\circ \Cd_X^+$ is identically $0$ on $(X\otimes S^{*,+})_{n,\sigma}$. In the latter case, we have the complex of finite-dimensional $G$-representations:
\begin{equation*}
0\longrightarrow \ker D^+_{X,n,\sigma}\longrightarrow (X\otimes S^{*,+})_{n,\sigma}\xrightarrow{D^+_X} (X\otimes S^{*,-})_{n,\sigma}\xrightarrow{D^-_X} \im D^-_{X,n,\sigma}\longrightarrow 0.
\end{equation*}
The claim in the proposition follows by the Euler-Poincar\' e principle.
\end{proof}

\begin{corollary}\label{c:local-index}
Suppose that 
\begin{enumerate}
\item $X$ is a subquotient of a standard module in $\Co_{t,c}(G,\fh)$, when $t\neq 0$;
\item $X$ is any graded $\cha$-module with infinitesimal character, when $t=0$.
\end{enumerate}
Then, the index formula from Proposition \ref{p:local-index} applies to $X$.
\end{corollary}

\begin{proof}
The corollary follows by noticing first that, when $t\neq 0$, every standard module satisfies the infinitesimal character condition. Secondly, since $\be\bu$ acts semisimply on a standard module, every subquotient also satisfies the infinitesimal character condition.
\end{proof}

\begin{corollary}\label{c:char-formula}
If $X$ has an infinitesimal character (in particular, $X$ could be $L_{t,c}(\tau)$), then the graded $G$-character of $X$ equals
\[\textup{ch}(X,g,\bq)=\frac{\textup{ch}(I_D(X),g,\bq)}{\det_{\fh^*}(1-g \bq)},\quad g\in G.
\]
\end{corollary}
Here $\bq$ is an indeterminate that we use to keep track of the degrees.
\begin{proof}
This is immediate from Proposition \ref{p:local-index}, since   $\textup{ch}({\bigwedge}^{\!+}\fh^*-{\bigwedge}^{\!-}\fh^*,g,\bq)=\det_{\fh^*}(1-g \bq)$.
\end{proof}

\subsection{Dirac cohomology and infinitesimal characters} The main general result about the Dirac cohomology $H_D(X)$ of a module $X$ is that it determines the infinitesimal character of $X$. In the setting of rational Cherednik algebras, this is proved in \cite{Ci}, and it is the analogue of the theorem proved by Huang-Pand\v zi\' c (see \cite{HP}) and conjectured by Vogan \cite{Vo} for $(\fg,K)$-modules of real reductive groups.

\begin{theorem}[{\it cf.} \cite{Ci}]\label{t:vogan-conj} Let $\sigma,\tau$ be two irreducible $G$-representations. If $\ghom_{G}(\sigma,H_D(L_{t,c}(\tau))\neq 0$ then $\tau\sim\sigma$.
\end{theorem}
\begin{proof}
When $t\neq 0$, this claim follows simply from Lemma \ref{l:D^2-sigma}. When $t=0$, the proof is nontrivial and this is done in \cite[Theorem 5.8 and Corollary 5.10]{Ci} in the ungraded setting. The proof in the graded setting is identical. 
\end{proof}

\subsection{Character formulae} We can refine  the character formula from Corollary \ref{c:char-formula}. 

\begin{proposition}\label{p:necessary}
Let $\tau$ be an irreducible $G$-representation. There exist finite-dimensional graded $G$-representations $i_{t,c}(\tau)^+$ and $i_{t,c}(\tau)^-$ such that
\begin{equation*}
\textup{ch}(L_{t,c}(\tau),g,\bq)=\frac{\textup{ch}(i_{t,c}(\tau)^+,g,\bq)-\textup{ch}(i_{t,c}(\tau)^-,g,\bq)}{\det_{\fh^*}(1-g \bq)},\quad g\in G,
\end{equation*}
and:
\begin{enumerate}
\item $\Hom_G(i_{t,c}(\tau)^+_d,i_{t,c}(\tau)^-_d)=0$ for each degree $d$;
\item If $\sigma$ is an irreducible $G$-representation which occurs in $i_{t,c}(\tau)^+$ or $i_{t,c}(\tau)^-$, then $\sigma\sim\tau$;
\item If $L_{t,c}(\tau)$ is finite dimensional\footnote{This is always the case when $t=0$.}, then $\dim i_{t,c}(\tau)^+=\dim i_{t,c}(\tau)^-$ as ungraded representations, and moreover, if $r$ is even, then
\[\textup{ch}(i_{t,c}(\tau)^{\pm,*}\otimes{{\det}_\fh},g,\bq)= \bq^{-r} \textup{ch}(i_{t,\check c}(\tau^*)^{\pm},g,\bq), 
\]
while if $r$ is odd, then
\[\textup{ch}(i_{t,c}(\tau)^{\pm,*}\otimes{{\det}_\fh},g,\bq)= \bq^{-r} \textup{ch}(i_{t,\check c}(\tau^*)^{\mp},g,\bq).
\]
\end{enumerate}
\end{proposition}

\begin{proof}
Since the Dirac index is a virtual finite-dimensional graded $G$-module, we may write it as \[I_D(L_{t,c}(\tau))=i_{t,c}(\tau)^+-i_{t,c}(\tau)^-,\] where $i_{t,c}(\tau)^\pm$ are finite-dimensional graded $G$-representations satisfying (1). Part (2) follows from Theorem \ref{t:vogan-conj}, since if $\sigma$ occurs in the Dirac index, it must occur in the Dirac cohomology. The claim about dimensions in (3) is obvious since $I_D(L_{t,c}(\tau))=L_{t,c}(\tau)\otimes S^{*,+}-L_{t,c}(\tau)\otimes S^{*,-}$, and $S^{*,+}$, $S^{*,-}$ have the same ungraded dimension. The last claim follows by tensoring with ${\det}_{\fh^*}$ in the formula $L_{t,c}(\tau)\otimes (S^{*,+}-\otimes S^{*,-})=i_{t,c}(\tau)^+-i_{t,c}(\tau)^-$ and using that
\[\textup{ch}((S^{*,+}-\otimes S^{*,-})\otimes \det\nolimits_{\fh^*},g,\bq)=(-1)^r \bq^{-r}\textup{ch}((S^{+}-\otimes S^{-}),g,\bq).
\]
This finishes the proof.
\end{proof}

\subsection{Composition numbers}
The Dirac index of a standard module is easy to determine.
\begin{lemma}\label{l:index-std}
$I_D(M_{t,c}(\tau))=\tau$, as graded $G$-modules.
\end{lemma}

\begin{proof}
As graded $G$-representations $M_{t,c}(\tau)=\mathbb C[\fh]\otimes\tau$. By Proposition \ref{p:local-index}:
\[I_D(M_{t,c}(\tau))=\mathbb C[\fh]\otimes \left({\bigwedge}^{\!+}\fh^*-{\bigwedge}^{\!-}\fh^*\right)\otimes\tau.
\]
The claim follows since the graded $G$-character of $\mathbb C[\fh]$ is $g\mapsto \frac 1{\det_{\fh^*}(1-g \bq)}$. 
\end{proof}
If $\tau$ and $\sigma$ are two irreducible $G$-representations, define the \emph{multiplicity polynomials} 
\begin{equation*}
n_{\tau,\sigma}(\bq)=\text{ the graded multiplicity of }L_{t,c}(\sigma)\text{ in }M_{t,c}(\tau).
\end{equation*}
We introduce the \emph{Dirac index polynomials} 
\begin{equation}\label{e:DiracIndPoly}
d_{\sigma,\tau}(\bq)=\text{ the graded multiplicity of }\tau \text{ in }I_D(L_{t,c}(\sigma)).
\end{equation}

\begin{proposition}\label{p:DiracPolyInverse}
The matrix of Dirac index polynomials $[d_{\sigma,\sigma'}(\bq)]$ is inverse to the matrix of multiplicity polynomials $[n_{\sigma,\sigma'}(\bq)]$.
\end{proposition}

\begin{proof}
In the Grothendieck group, we write
\[M_{t,c}(\tau)=\sum_{\sigma}n_{\tau,\sigma}(\bq) L_{t,c}(\sigma).
\]
Apply the Dirac index, which by Proposition \ref{p:local-index}, is additive:
\begin{equation*}
I_D(M_{t,c}(\tau))=\sum_{\sigma}n_{\tau,\sigma}(\bq) I_D(L_{t,c}(\sigma)).
\end{equation*}
From Lemma \ref{l:index-std} and the definition of the Dirac index polynomials, we get
\begin{equation*}
\tau=\sum_{\sigma}n_{\tau,\sigma}(\bq)\sum_{\sigma'}d_{\sigma,\sigma'}(\bq)\sigma',
\end{equation*}
hence 
\[\sum_{\sigma}n_{\tau,\sigma}(\bq)d_{\sigma,\sigma'}(\bq)=\delta_{\tau,\sigma'},
\]
which proves the claim.
\end{proof}

\begin{remark}
The similar results (and proofs) hold for the restricted algebra $\reschazero$. The difference is that 
\begin{equation*}
I_D(\bar M_{t,c})=P_G(\bq)\cdot \tau,\quad\text{ where } P_G(\bq)=\prod_i(1-\bq^{d_i}),
\end{equation*}
with $d_i$ being the fundamental degrees of $G$. This is because the graded character of $\mathbb C[\fh]_G$ is $g\mapsto \frac {P_G(\bq)}{\det_{\fh^*}(1-g \bq)}$. Consequently, in $\reschazero$,
\begin{equation}\label{e:multi-restricted}
[d_{\sigma,\sigma'}(\bq)]\cdot [n_{\sigma,\sigma'}(\bq)]=P_G(\bq)\cdot\textup{Id}.
\end{equation}
\end{remark}

\subsection{The Euler-Poincar\' e pairing} Let $X$ be an $\cha$-module. The {\it Koszul-type} resolution of $X$ is (\cite[proof of Theorem 1.8]{EG}, also \cite[section 3.2]{Cha} for graded affine Hecke algebras)
\begin{equation}\label{e:koszul}
0\leftarrow X\xleftarrow{\epsilon} \cha\otimes_{\bbc G} X\mid_G\xleftarrow{d} \cha\otimes_{\bbc G} (V\otimes X\mid_G)\xleftarrow{d} \cha\otimes_{\bbc G} \left({\bigwedge}^{\!2}V\otimes X\mid_G\right)\leftarrow\cdots.
\end{equation}
The map $\epsilon$ is given by the $\cha$-action: $h\otimes x\mapsto h\cdot x$. The maps $d$ are given by
\begin{equation*}
\begin{aligned}
d:\  h\otimes v_1\wedge\dots&\wedge v_p\otimes x\mapsto \sum_{j=1}^p (-1)^{j-1} (h v_j\otimes v_1\wedge\dots\wedge \hat v_j\wedge\dots\wedge v_p\otimes x\\
&-h\otimes v_1\wedge\dots\wedge \hat v_j\wedge\dots\wedge v_p\otimes v_j\cdot x).
\end{aligned}
\end{equation*}
\begin{proposition}[\cite{EG}]
In (\ref{e:koszul}), $$d^2=0,$$ and the complex is a projective resolution of $X$.
\end{proposition}
\begin{proof}
This is proved in \cite[Theorem 1.8(i)]{EG}: the identity $d^2=0$ is proved by a direct calculation, while the exactness of the complex by the standard technique of passing to the associated graded complex with respect to the filtration where the elements of $V$ are given degree $1$.
\end{proof}
Let $Y$ be another $\cha$-module. To compute $\Ext_{\cha}^\bullet(X,Y)$, apply $\Hom_{\cha}(-,Y)$ to the resolution (\ref{e:koszul}) and take the cohomology of the resulting complex. We do this in the graded setting, where $\cha$ comes with the grading from before. Regard $V=\fh\oplus\fh^*$ as a graded $G$-representation, where $\fh$ has degree $-1$ and $\fh^*$ has degree $1$ (same as in $\cha$). 

Assume that $X$ and $Y$ are graded $\cha$-modules which are finite dimensional in each degree. Notice that the differential $d$ preserves degrees, {\it i.e.}, it maps $(\cha\otimes_{\bbc G} ({\bigwedge}^{\!p}V\otimes X\mid_G))_n$ to $(\cha\otimes_{\bbc G} ({\bigwedge}^{\!p-1}V\otimes X\mid_G))_n$ for each degree $n$. Hence, (\ref{e:koszul}) is a projective resolution in the graded complex. Apply $\ghom_{\cha}(-,Y)$ and get the complex:
\begin{equation*}
d^*: \ghom_{\cha}\left(\cha\otimes_{\bbc G} \left({\bigwedge}^{\!p} V\otimes X\mid_G\right),Y\right)\rightarrow \ghom_{\cha}\left(\cha\otimes_{\bbc G} \left({\bigwedge}^{\!p+1}V\otimes X\mid_G\right),Y\right),
\end{equation*}
which, by Frobenius reciprocity, is equivalent to
\begin{equation}\label{e:G-koszul}
d^*: \ghom_{\bbc G}\left({\bigwedge}^{\!p} V\otimes X\mid_G,Y\mid_G\right)\rightarrow \ghom_{\bbc G} \left({\bigwedge}^{\!p+1}V\otimes X\mid_G,Y\mid_G\right).
\end{equation}
Here, the induced maps $d^*$ are given by
\begin{multline*}
d^*(\phi)(v_1\wedge\dots\wedge v_{p+1}\otimes x)=\sum_{i=1}^{p+1} (-1)^{i+1} (v_i\cdot \phi(v_1\wedge\dots\wedge\hat v_i \wedge\dots\wedge v_{p+1}\otimes x)\\
-\phi(v_1\wedge\dots\wedge\hat v_i \wedge\dots\wedge v_{p+1}\otimes v_i\cdot x) ),
\end{multline*}
where $\phi\in \ghom_{\bbc G}({\bigwedge}^{\!p} V\otimes X\mid_G,Y\mid_G).$

\begin{definition}
Let $X,Y$ be graded $\cha$-modules of finite type. The \emph{graded Euler-Poincar\' e pairing} is
\[\textup{EPgr}_{\cha}(X,Y)=\sum_{p\ge 0}(-1)^p\textup{dimgr}\Ext^p_{\cha}(X,Y).
\]
By the previous discussion, this sum is finite and well defined.
\end{definition}

\begin{corollary}\label{c:EP}
For $X,Y$ graded $\cha$-modules of finite type,
\begin{equation}\label{e:EP}
\textup{EPgr}_{\cha}(X,Y)=\dim\ghom_{\bbc G}\left(\left({\bigwedge}^{\!+} V-{\bigwedge}^{\!-} V\right)\otimes X\mid_G,Y\mid_G\right).
\end{equation}
\end{corollary}
\begin{proof}
This is immediate by the Euler-Poincar\'e principle applied in (\ref{e:G-koszul}).
\end{proof}

\subsection{Elliptic pairings} Consider the Grothendieck group of graded $G$-representations $R(G)_{\textup{gr}}$ with the graded character pairing
\[\langle \sigma,\sigma'\rangle^{\textup{gr}}_G=\dim \ghom_G(\sigma,\sigma').
\]
Define the $V$-\emph{elliptic pairing} 
\begin{equation}
\langle \sigma,\sigma'\rangle^{V-\textup{ell}}_G=\left\langle \sigma\otimes\left({\bigwedge}^{\!+} V-{\bigwedge}^{\!-} V\right),\sigma'\right\rangle^{\textup{gr}}_G.
\end{equation}
The previously defined notions can be related as follows:

\begin{theorem}\label{t:pairs} Suppose that $X,Y$ are graded $\cha$-modules of finite type. 
\begin{enumerate}
\item $\textup{EPgr}_{\cha}(X,Y)=\langle X\mid_G, Y\mid_G\rangle^{V-\textup{ell}}_G.$
\item If in addition, $X$ and $Y$ have infinitesimal characters, then
\begin{equation}\label{e:pairs}
\textup{EPgr}_{\cha}(X,Y)=\langle X\mid_G, Y\mid_G\rangle^{V-\textup{ell}}_G=\langle I_D(X),I_D(Y)\rangle^{\textup{gr}}_G.
\end{equation}
\end{enumerate}
\end{theorem}

\begin{proof}
The first formula follows immediately from the definitions and Corollary \ref{c:EP}. For (\ref{e:pairs}), notice first that 
\[{\bigwedge}^{\!+} V-{\bigwedge}^{\!-} V=\left({\bigwedge}^{\!+} \fh-{\bigwedge}^{\!-} \fh\right)\otimes \left({\bigwedge}^{\!+} \fh^*-{\bigwedge}^{\!-} \fh^*\right),
\]
hence
\[ \langle X\mid_G,Y\mid_G\rangle^{V-\textup{ell}}_G=\langle X\otimes(S^{*,+}-S^{*,-}),Y\otimes(S^{*,+}-S^{*,-})\rangle^{\textup{gr}}_G=\langle I_D(X),I_D(Y)\rangle^{\textup{gr}}_G,
\]
by Corollary \ref{c:local-index}.
\end{proof}

\begin{corollary}\label{c:EPupshot}
For every irreducible $G$-representations $\sigma$ and $\tau$,
\[\textup{EPgr}_{\cha}(M_{t,c}(\tau),M_{t,c}(\sigma))=\delta_{\tau,\sigma}
\]
and
\begin{equation*}
\textup{EPgr}_{\cha}(M_{t,c}(\tau),L_{t,c}(\sigma))=d_{\sigma,\tau}(\bq),
\end{equation*}
where $d_{\sigma,\tau}(\bq)$ is the Dirac index polynomial.
\end{corollary}

\begin{proof}
Immediate from Theorem \ref{t:pairs}, Lemma \ref{l:index-std}, and the definition of $d_{\sigma,\tau}(\bq).$
\end{proof}

\subsection{An example: $t=0$, $G=W(B_2)$}
For $\reschazero$ of type $B_2$ and with constant parameter $c$, there is an interesting CM family consisting of the irreducible $W(B_2)$-representations $\{11\times 0,0\times 2,1\times 1\}$ (here we use Carter's notation via bipartitions for the irreducible $W(B_n)$-representations). This is a particular case of the cuspidal CM families studied in \cite{BT} and \cite{Ci2}.
The graded $W(B_2)$-structures of the corresponding simple modules are:
\begin{equation*}
\begin{aligned}
\bar L(11\times 0)&=11\times 0\\
\bar L(0\times 2)&=0\times 2\\
\bar L(1\times 1)&=(1+\bq^2)(1\times 1)+\bq (2\times 0+0\times 11).
\end{aligned}
\end{equation*}
These are obtained as follows. The modules $\bar L(11\times 0)$ and $\bar L(0\times 2)$ can be constructed by requiring $\fh$ and $\fh^*$ to act identically by zero, see {\it loc. cit.}. Then, to determine the graded character of $\bar L(1\times 1)$, we can look at the graded characters of $\bar M(11\times 0)$ and $\bar M(0\times 2)$, for example.

\smallskip

The graded character of ${\bigwedge}^{\!+}\fh-{\bigwedge}^{\!-}\fh$ is $2\times 0-\bq(1\times 1)+\bq^2(0\times 11)$. Applying Proposition \ref{p:local-index}, we find that 
the matrix of Dirac polynomials for the cuspidal block of $\{11\times 0,0\times 2,1\times 1\}$ in this order is:
\begin{equation*}
[d_{\sigma,\tau}(\bq)]=\left(\begin{matrix} 1 & \bq^2 &-\bq\\ \bq^2 &1&-\bq \\ -(\bq+\bq^3) & -(\bq+\bq^3) &1+\bq^4\end{matrix}\right).
\end{equation*}
As an illustration of formula (\ref{e:multi-restricted}), the matrix of multiplicity polynomials for this block is
\begin{equation*}
[n_{\sigma,\tau}(\bq)]=P_{W(B_2)}(\bq)\cdot [d_{\sigma,\tau}(\bq)]^{-1}=\left(\begin{matrix} 1&\bq^4 &\bq\\ \bq^4 &1 &\bq\\ \bq+\bq^3 & \bq+\bq^3 &1+\bq^2
\end{matrix}
\right);
\end{equation*}
here $P_{W(B_2)}(\bq)=(1-\bq^2)(1-\bq^4)$.

\subsection{Another example: $t=0$, $G=W(G_2)$} We use Carter's notation for irreducible representations of $W(G_2)$. These are: $\phi_{1,0}$, $\phi_{1,3}'$, $\phi_{1,3}''$, $\phi_{1,6}$, $\phi_{2,1}$, and $\phi_{2,2}$, where the first number is the dimension of the representation and the second is the lowest harmonic degree. For $\reschazero$ of type $G_2$ and with constant parameter $c$, there is a cuspidal CM family consisting of the irreducible $W(G_2)$-representations $\{\phi_{1,3}',\phi_{1,3}'',\phi_{2,2},\phi_{2,1}\}$. The graded $W(G_2)$-structures of the corresponding simple modules are:
\begin{equation*}
\begin{aligned}
\bar L(\phi_{1,3}')&=\phi_{1,3}',\\
\bar L(\phi_{1,3}'')&=\phi_{1,3}'',\\
\bar L(\phi_{2,2})&=\phi_{2,2},\\
\bar L(\phi_{2,1})&=(1+\bq^2)\phi_{2,1}+\bq (\phi_{1,0}+\phi_{1,6}).
\end{aligned}
\end{equation*}
 The modules $\bar L(\phi_{1,3}')$, $\bar L(\phi_{1,3}'')$, $\bar L(\phi_{2,2})$ can be constructed by requiring $\fh$ and $\fh^*$ to act identically by zero, see \cite{BT} and \cite{Ci2}. The structure of $\bar L(\phi_{2,1})$ can be obtained by analysing the structure of the standard module $\bar M(\phi_{1,3}')$. The graded character of ${\bigwedge}^{\!+}\fh-{\bigwedge}^{\!-}\fh$ is $\phi_{1,0}-\bq \phi_{2,1}+\bq^2\phi_{1,6}$. Applying Proposition \ref{p:local-index}, we find that 
the matrix of Dirac polynomials for the cuspidal block of $\{\phi_{1,3}',\phi_{1,3}'',\phi_{2,2},\phi_{2,1}\}$ in this order is:
\begin{equation*}
[d_{\sigma,\tau}(\bq)]=\left(\begin{matrix} 1 & \bq^2 &-\bq &0\\ \bq^2 &1&-\bq &0 \\ -\bq &-\bq &(1+\bq^2) &-\bq\\
0&0 & -(\bq+\bq^3) &1+\bq^4\end{matrix}\right).
\end{equation*}
As another illustration of formula (\ref{e:multi-restricted}), the matrix of multiplicity polynomials for this block is
\begin{equation*}
[n_{\sigma,\tau}(\bq)]=P_{W(G_2)}(\bq)\cdot [d_{\sigma,\tau}(\bq)]^{-1}=\left(\begin{matrix} 1 &\bq^6 &\bq+\bq^5 &\bq^2\\ \bq^6 &1 &\bq+\bq^5 &\bq^2\\ \bq+\bq^5 &\bq+\bq^5 &(1+\bq^2)(1+\bq^4) &\bq+\bq^3\\
\bq^2+\bq^4 &\bq^2+\bq^4 &\bq (1+\bq^2)^2 &(1+\bq^4)
\end{matrix}
\right);
\end{equation*}
here $P_{W(G_2)}(\bq)=(1-\bq^2)(1-\bq^6)$.

\section{integral-reflection representations}\label{s:IRR}

We construct, for the rational Cherednik algebra, analogues of the integral-reflection representations as in the theory of affine Hecke algebras and trigonometric Cherednik algebras \cite{HO}, \cite{EOS}.

\subsection{Integral operators} In this subsection, we extend the ideas of \cite{Gu} to the case of a complex reflection group. Denote by $\lpi\cdot,\cdot\rpi:\bbc[\fh]\times\bbc[\fh^*]\to\bbc$ the bilinear pairing defined by
\begin{equation}\label{e:polyduality}
\lpi p,q \rpi = \big(p(\partial)q\big)(0),
\end{equation}
for all $p\in \bbc[\fh]$ and $q\in \bbc[\fh^*]$. Here, an element $x\in \fh^*$ acts on $\fh$ by the derivative $\partial_x(y) = x(y)$, which is extended to any $q\in\bbc[\fh^*]$ by linearity and the Leibniz rule. We shall use the notation $x(\partial) = \partial_x$. We also have a similarly defined operator $\partial_y$ on $\bbc[\fh]$ for any $y\in\fh$. Note that the bilinear pairing above is an extension to $\Sym(\fh^*)\times \Sym(\fh)  = \bbc[\fh]\times\bbc[\fh^*]$ of the duality between $\fh^*\times \fh\to\bbc$. The following lemma is well known:

\begin{lemma}
If $x\in\fh^*$ and $y\in\fh$, let $\mu_x:\bbc[\fh]\to \bbc[\fh]$ and $\mu_y:\bbc[\fh^*]\to\bbc[\fh^*]$ denote the corresponding multiplication operators. Then $\mu_x$ is dual $\partial_x$ and $\partial_y$ is dual to $\mu_y$, under the pairing $\lpi\cdot,\cdot\rpi$ of (\ref{e:polyduality}).
\end{lemma}

\begin{definition}
For each $s\in \Cs$ and choice of eigenvectors $\{\alpha_s,\alpha_s^\vee\}$ with $\alpha_s^\vee(\alpha_s) = 2$, define the \emph{integral operators} $I_s=I(s,\alpha_s):\bbc[\fh]\to\bbc[\fh]$ and $I^{\vee}_s = I(s,\alpha_s^\vee):\bbc[\fh^*]\to\bbc[\fh^*]$ by
\begin{equation}\label{e:integraloperators}
\begin{array}{rcl}
(I_sp)(y)&=& \displaystyle\int\limits_{-d_{s^{-1}}^\vee \alpha_s(y)}^0p(y+t\alpha_s^\vee)dt \\
(I^{\vee}_sq)(x)&=& \displaystyle\int\limits_{-d_{s^{-1}}x(\alpha_s^\vee)}^0q(x+t\alpha_s)dt, 
\end{array}
\end{equation}
where $d_{s^{-1}}$ and $d_{s^{-1}}^\vee$ are as in Lemma \ref{l:reflections}. 
\end{definition}

We shall typically leave implicit the choice of pair of eigenvectors for $s$. Any other choice $\{\mu\alpha_s,\mu^{-1}\alpha_s^{\vee}\}$ scales the integral operators, $I(s,\mu\alpha_s) = \mu I(s,\alpha_s)$ and $I(s,\mu^{-1}\alpha^\vee_s) = \mu^{-1} I(s,\alpha^\vee_s)$. Similarly to the heuristics of Gutkin \cite{Gu}, we can think of these operators as the line integral from $s^{-1}(y)$ to $y$, respectively $s^{-1}(x)$ to $x$:
\[(I_sp)(y) = \int\limits_{s^{-1}(y)}^{y}p\qquad\textup{and}\qquad (I^{\vee}_sq)(x) = \int\limits_{s^{-1}(x)}^{x}q.\]

\begin{lemma}\label{l:intops}
For all $x\in\fh^*$ and $y\in\fh$, the integral operators defined above satisfy
\begin{enumerate}
\item $I_s\circ \partial_{\alpha_s^\vee} = 1 - s$,
\item $[\partial_y,I_s] = d_s\alpha_s(y) s$,
\item $I^{\vee}_s\circ \partial_{\alpha_s} = 1 - s$,
\item $[\partial_x,I^{\vee}_s] = d^\vee_sx(\alpha_s^\vee) s$.
\end{enumerate}
\end{lemma}

\begin{proof}
Our arguments are similar to those of \cite{Gu}. For the first two identities, pick a pair of basis $(\{x_j\},\{y_j\})$ in duality with the property that $x_1 = \alpha_s/2$ and $s(x_j) = x_j$, if $j>1$. Note that $y_1 = \alpha_s^\vee$. Write $\eta_j = x_j(y)$, so that we can write $y = (\eta_1,\ldots,\eta_r)$, with respect to  $(x_1,\ldots,x_r)$, where $r = \dim\fh$. Then,
\[(I_sp)(y)= (I_sp)(\eta_1,\ldots,\eta_r) =\displaystyle\int\limits_{-(1-\lambda_s)\eta_1}^0p(\eta_1+t,\eta_2,\ldots,\eta_r)dt=\displaystyle\int\limits_{\lambda_s\eta_1}^{\eta_1}p(t,\eta_2,\ldots,\eta_r)dt\]
after a change of variables. Since $\partial_{y_j}$ is the $j$-th partial derivative in these coordinates, it is then clear that $[\partial_{y_j},I_s] = 0$, if $j>1$. This and  $s(y_j) = y_j$ for $j>1$ implies
\[
(I_s\circ \partial_{\alpha_s^\vee} (p))(y) =p(\eta_1,\eta_2\ldots,\eta_r) - p(\lambda_s\eta_1,\eta_2,\ldots,\eta_r) = p(y) - p(s^{-1}(y)),
\]
from which we conclude $I_s\circ \partial_{\alpha_s^\vee} = 1 - s$, establishing (1). Similarly, $\partial_{\alpha_s^\vee}\circ I_s = 1-\lambda_s s$. If we write $y = x_1(y)y_1 + \sum_{j>1} x_j(y)y_j$, it follows from the above that
\[[\partial_y,I^{\vee}_s] = \frac{\alpha_s(y)}{2}[\partial_{\alpha_s^\vee},I_s] = \alpha_s(y)\frac{1-\lambda_s}{2} s,\]
proving (2). 

For the last identities, we use the pair of basis $((y_j),(x_j))$ in duality with $y_1 = \alpha^\vee_s/2$, $s(y_j) = y_j$ for $j>1$, which implies $x_1 = \alpha_s$. The integral operator can be written as 
\[(I^{\vee}_sq)(x)= \displaystyle\int\limits_{\lambda^{-1}_s\xi_1}^{\xi_1}p(t,\xi_2,\ldots,\xi_r)dt,\]
where $x = (\xi_1,\ldots,\xi_r)$ with respect to the coordinates $(y_1,\ldots, y_r)$. The rest is similar.
\end{proof}

\begin{proposition}\label{p:calcduality}
Via the pairing $\lpi\cdot,\cdot\rpi$ of (\ref{e:polyduality}), the operators $\Delta(s,\alpha_s),I(s,\alpha_s):\bbc[\fh]\to\bbc[\fh]$ are dual to $I^\vee(s^{-1},\alpha_s^\vee), \Delta(s^{-1},\alpha_s^\vee):\bbc[\fh^*]\to\bbc[\fh^*]$, respectively.
\end{proposition}
\begin{proof} If $A\in\End(\bbc[\fh])$, let us denote by $A^\bullet$ the uniquely defined endomorphism of $\bbc[\fh^*]$ that satisfies $\lpi Ap,q\rpi = \lpi p, A^\bullet q\rpi$. We note that $s^\bullet = s^{-1}$ for any $s\in\Cs$. That said, we follow the arguments of \cite{Gu} in the case of real reflection groups. The operator $\Delta_s = \Delta(s,\alpha_s)$ is defined by the equation
\[1 - s = \mu_{\alpha_s}\circ\Delta(s,\alpha_s)\]
in $\End(\bbc[\fh])$. We note that this equation is independent of the choice of $\alpha_s$. Therefore, the operator dual to $\Delta(s,\alpha_s)$ must verify the equation
\[1 - s^{-1} = \Delta(s,\alpha_s)^\bullet \circ \partial_{\alpha_s}\]
in $\End(\bbc[\fh^*])$. From Lemma \ref{l:intops}, it follows that $I(s^{-1},\alpha_s^\vee)$ satisfy this equation, and, since $\partial_{\alpha_s}$ is surjective, we conclude that $I(s^{-1},\alpha_s^\vee)$ is dual to $\Delta(s,\alpha_s)$. The proof that $I(s,\alpha_s)$ is dual to $\Delta(s^{-1},\alpha_s^\vee)$ is similar.
\end{proof}

\begin{remark}
If we assume that the choices of $\{\alpha_s,\alpha_s^\vee\}$ are such that $\alpha_{s^j} = \alpha_s$, for all $j\in\bbz$ and all reflections, we have that $\Delta_s,I_{s}$ are dual to $I_{s^{-1}}^\vee,\Delta_{s^{-1}}^\vee$. We assume this to be the case in what follows.
\end{remark}

\begin{proposition}\label{p:conjintegralop}
For any $g\in G$, $s\in\Cs$ and eigenvectors $\{\alpha_s,\alpha^\vee_s\}$ such that $\alpha^\vee_s(\alpha_s) = 2$, we have
\[gI(s,\alpha_s)g^{-1} = I(gsg^{-1},g(\alpha_s))\quad\textup{and}\quad gI(s,\alpha^\vee_s)g^{-1} = I(gsg^{-1},g(\alpha^\vee_s)).\]
\end{proposition}

\begin{proof}
For any $p\in\bbc[\fh]$ and $y\in \fh^*$, we have
\begin{align*}
g\Big(I(s,\alpha_s)(g^{-1}(p))\Big)(y) &= \displaystyle\int\limits_{-d_{s^{-1}}^\vee \alpha_s(g^{-1}(y))}^0(g^{-1}p)(g^{-1}(y)+t\alpha_s^\vee)dt\\
&= \Big(I(gsg^{-1},g(\alpha_s))(p)\Big)(y),
\end{align*}
as required. The proof for $I(gsg^{-1},g(\alpha_s^\vee))$ is similar.
\end{proof}

\begin{proposition}\label{p:s-expansion}
A reflection $s\in\Cs$ has an expansion in the formal completion of the Weyl algebra acting on $\bbc[\fh]$ as
\[s = \sum_{n\geq 0}\frac{(-1)^n}{n!}d_s^n\mu_{\alpha_s}^{n}\circ\partial_{\alpha^\vee_s}^{n}.\]
\end{proposition}

\begin{proof}
Let $p\in\bbc[\fh]$. We shall show that
\begin{equation}\label{e:TBP1}
s(p) = \sum_{n\geq 0}\frac{(-d_s)^n}{n!}\alpha_s^{n}(\partial_{\alpha^\vee_s}^{n}p),
\end{equation}
which is a finite sum. The proof will be by induction on $d = \deg(p)$. By linearity, it is sufficient to assume that $p$ is a monomial. From Lemma \ref{l:reflections}, the result is true for $d=1$. Now suppose (\ref{e:TBP1}) holds for a monomial $p$ of degree $d$. Write $\partial = \partial_{\alpha_s^\vee}$ and note that
\begin{equation}\label{e:HighDirLeibniz}
\partial^n(px) = (\partial^np)x + n (\partial^{n-1} p)x(\alpha_s^\vee).
\end{equation}
Then, from $s(px) = s(p)s(x)$ and using (\ref{e:TBP1}), we compute
\begin{align*}
s(px) &= s(p)x - s(p)d_sx(\alpha_s^\vee)\alpha_s\\
&= \sum_{n\geq 0}\frac{(-d_s)^n}{n!}\alpha_s^{n}(\partial^{n}p)x + \sum_{n\geq 0}\frac{(-d_s)^{n+1}}{(n+1)!}(n+1)\alpha_s^{n+1}(\partial^{n}p)x(\alpha_s^\vee)\\
&=px + \sum_{m\geq 1}\frac{(-d_s)^m}{m!}\alpha_s^{m}\big((\partial^mp)x + m (\partial^{m-1} p)x(\alpha_s^\vee)\big),
\end{align*}
and the result follows from (\ref{e:HighDirLeibniz}).
\end{proof}

\begin{corollary}\label{c:operators-expansion}
The following expansions hold in the formal completion of the Weyl algebra acting on $\bbc[\fh]$:
\begin{align*}
(1-s) &= \sum_{n\geq 1}\frac{(-1)^{n+1}}{n!}d_s^n\mu^{n}_{\alpha_s}\circ\partial_{\alpha_s^\vee}^{n},\\
\Delta_s &= \sum_{n\geq 1}\frac{(-1)^{n+1}}{n!}d_s^n\mu^{n-1}_{\alpha_s}\circ\partial_{\alpha_s^\vee}^{n},\\
I_s &= \sum_{n\geq 1}\frac{(-1)^{n+1}}{n!}d_s^n\mu^{n}_{\alpha_s}\circ\partial_{\alpha_s^\vee}^{n-1}.
\end{align*}
\end{corollary}

\begin{proof} It is immediate from Proposition \ref{p:s-expansion} and from $\mu_{\alpha_s}\circ\Delta_s = 1-s = I_s\circ\partial_{\alpha^\vee_s}$.
\end{proof}
As indicated in the proof of Proposition \ref{p:s-expansion}, the expansions above are well defined in $\End(\bbc[\fh])$, as they end up being a finite sum. By exchanging $\alpha_s$ and $\alpha^\vee_s$ (and using $d^\vee_s$ instead of $d_s$), similar formulae for the actions on $\bbc[\fh^*]$ also hold. 

We collect some other simple facts that follow from the definition of $I^\vee_s$ (similarly for $I_s$): 
\begin{lemma}\label{l:observation1}
Given a reflection $s$ and $q\in \bbc[\fh^*]$, then $I^\vee_s q = (1-s)\tilde{q}$, where $\tilde{q}$ is any polynomial such that $\partial_{\alpha_s}\tilde{q} = q$. 
\end{lemma}
\begin{proof}
It is clear from  Lemma \ref{l:intops}, as $I^\vee_s\circ \partial_{\alpha_s} = 1-s$. Such $\tilde{q}$ always exists as $\partial_{\alpha_s}$ is surjective.
\end{proof}

\begin{lemma}\label{l:observation2}
If $\partial_{\alpha_s}q = 0$, then $s(q) = q$. In particular, $I_s^\vee q = d_s^\vee\alpha^\vee_s q$.
\end{lemma}
\begin{proof}
Choose a basis of $\fh$ so that $y_1 = \alpha_s^\vee/2$ and $s(y_j) = y_j$ for $j>1$. Then, $\partial_{\alpha_s}q = 0$ means that $q$ is a polynomial in $y_2,\ldots,y_r$, from which $s(q) = q$. For the in particular, pick $\tilde q = \frac{\alpha_s^\vee}{2}q$ and the result follows from Lemma \ref{l:observation1}.
\end{proof}
\begin{lemma}[Integration by parts] For any $q_1,q_2\in \bbc[\fh^*]$, it holds that \[I_s^\vee(q_1(\partial_{\alpha_s}q_2)) = (1-s)(q_1q_2) - I_s^\vee((\partial_{\alpha_s}q_1)q_2).\]
\end{lemma}
\begin{proof}
We have $I_s^\vee(\partial_{\alpha_s}(q_1q_2)) = (1-s)(q_1q_2)$, from Lemma \ref{l:intops}, and thus the result.
\end{proof}

\subsection{Dualised induced modules}
If $\tau^*\in\Rep(G)$, we shall consider the standard $\chad$-module
\[M_{t,\check{c}}(\tau^*) = \chad\otimes_{\bbc[\fh]\rtimes\bbc G}\tau^* = \Ind_{\bbc G}^{\chad} \tau^*\]
and the projective $\chad$-module
\[X_{t,\check{c}}(\tau^*) = \chad\otimes_{\bbc G}\tau^* = \Ind_{\bbc G}^{\chad} \tau^*.\]

As linear spaces, we have $M_{t,\check{c}}(\tau^*) = \bbc[\fh^*] \otimes \tau^*$ and $X_{t,\check{c}}(\tau^*) = \bbc[\fh^*]\otimes \bbc[\fh] \otimes \tau^*$. Denote by $\lpi\cdot,\cdot\rpi$ the nondegenerate bilinear pairing between $\bbc[\fh^*]\otimes \bbc[\fh] \otimes \tau^*$ and $\bbc[\fh]\otimes \bbc[\fh^*] \otimes \tau$ given by the products of the pairings in (\ref{e:polyduality}) and the natural pairing between $\tau^*$ and $\tau$:
\begin{equation}\label{e:Xduality}
\lpi q\otimes p \otimes z^*, p'\otimes q' \otimes z\rpi = \big((q(\partial)p')(0)\big)\big((p(\partial)q')(0)\big)z^*(z),
\end{equation}
for all $p,p'\in\bbc[\fh]$, $q,q'\in \bbc[\fh^*]$ and $(z^*,z)\in \tau^*\times \tau$. We shall also denote by $\lpi\cdot,\cdot\rpi$ the similarly defined nondegenerate pairing between $\bbc[\fh^*]\otimes\tau^*$ and $\bbc[\fh]\otimes\tau$. Dualising the natural left multiplication action of $\chad$ using $(\lpi\cdot,\cdot\rpi,\gamma)$, we obtain modules for $\cha$.

\begin{theorem}\label{t:M-module}
The linear space $\fM_{t,c}(\tau)=\bbc[\fh]\otimes \tau$ endowed with operators $Q(x), Q(y)$ and $Q(s)$, for $x\in \fh^*$, $y\in\fh$ and $s\in \Cs$ defined by
\begin{align*}
Q(x)(p\otimes z)&=t\mu_x(p)\otimes z - \sum_s\frac{c(s)}{d_s}x(\alpha^\vee_s)I_s(p)\otimes s(z)\\
Q(y)(p\otimes z)&= \partial_y(p)\otimes z\\
Q(s)(p\otimes z)&= s(p)\otimes s(z)
\end{align*}
extends to a module for the Cherednik algebra $\cha$. As before, $d_s = \frac{1-\lambda_s}{2}$.
\end{theorem}

\begin{proof}
The $Q$-operators above are dual to the operators on $M_{t,\check c}(\tau^*)$ induced by the left multiplication on $\chad$. But, one can actually check the defining relations by hand. Note that, for all $y\in\fh,x\in\fh^*,p\in\bbc[\fh]$ and $z\in \tau$, we have
\begin{align*}
[Q(y),Q(x)](p\otimes z) &= t[\partial_y,\mu_x](p)\otimes z - \sum_s\frac{c(s)}{d_s}x(\alpha^\vee_s)[\partial_y,I_s](p)\otimes s(z)\\
&= tx(y)p\otimes z - \sum_sc(s)x(\alpha_s^\vee)\alpha_s(y)s(p)\otimes s(z)\\
&= Q([y,x])(p\otimes z),
\end{align*}
where we used Lemma \ref{l:intops}. For the other relations in (\ref{e:relations}), using that $s\circ \partial_y\circ s^{-1} = \partial_{s(y)}$ and $s\circ \mu_x\circ s^{-1} = \mu_{s(x)}$, one can easily show that $Q(s)Q(y)Q(s^{-1}) = Q(s(y))$ and that $Q(s)Q(x)Q(s^{-1}) = Q(s(x))$. It is also straightforward to show $[Q(y),Q(y')] = 0$. It is less straightforward but still true that $[Q(x),Q(x')] = 0$, as will be shown next. 
\end{proof}

\begin{proposition}\label{p:x-Commutation}
For all $x,x'\in\fh^*$, we have $[Q(x),Q(x')] = 0$.
\end{proposition}

\begin{proof}
First of all, we claim that there exists an operator $\Psi_s\in\End(\bbc[\fh])$ such that for any $x\in\fh$ and any reflection $s\in\Cs$ we have $[I_s,\mu_x] = x(\alpha_s^\vee)\Psi_s$. Indeed, let 
\[\Psi_s = \sum_{n\geq 2}\frac{(-1)^{n+1}}{n(n-2)!}d_s^n\mu_{\alpha_s}^n\circ \partial^{n-2}_{\alpha_s^\vee}.\]
Then, one computes using Corollary \ref{c:operators-expansion} that
\begin{align*}
[I_s,\mu_x] &= \sum_{n\geq 1}\frac{(-1)^{n+1}}{n!}d_s^n\big(\mu^{n}_{\alpha_s}\circ\partial_{\alpha_s^\vee}^{n-1}\circ\mu_x - \mu_x\circ\mu^{n}_{\alpha_s}\circ\partial_{\alpha_s^\vee}^{n-1}\big)\\
&= \sum_{n\geq 1}\frac{(-1)^{n+1}}{n!}d_s^n\mu^{n}_{\alpha_s}\circ[\partial^{n-1}_{\alpha_s^\vee},\mu_x]\\
&=x(\alpha_s^\vee)\Psi_s,
\end{align*}
where we used that $[\partial^{n-1}_{\alpha_s^\vee},\mu_x] = (n-1)x(\alpha_s^\vee)\partial^{n-2}_{\alpha_s^\vee}$. We then compute
\begin{multline*}
[Q(x),Q(x')] = t^2[\mu_x,\mu_{x'}]\otimes 1 - t\sum_s\frac{c(s)}{d_s}\big(x'(\alpha_s^\vee)[\mu_x,I_s] + x(\alpha_s)[I_s,\mu_{x'}]\big)\otimes s\\
+\sum_{s,r}\frac{c(s)c(r)}{d_sd_r}\big(x(\alpha_s^\vee)x'(\alpha_r^\vee)I_s\circ I_r\otimes sr -x'(\alpha_r^\vee)x(\alpha_s^\vee)I_r\circ I_s\otimes rs \big).
\end{multline*}
The first and last summands clearly vanish, as well as the middle one, since $[\mu_x,I_s] = -x(\alpha_s^\vee)\Psi_s$ and $[I_s,\mu_{x'}] = x'(\alpha_s^\vee)\Psi_s$.
\end{proof}

\begin{remark}
One can similarly show that there exists an operator $\Phi_s\in\End(\bbc[\fh])$ so that $[\Delta_s,\partial_y] = \alpha_s(y)\Phi_s$. One can then adapt the proof of the previous proposition to give an elementary proof of the commutativity of the Dunkl operators $T_y = \partial_y - \sum_s\frac{c(s)}{d_s}\alpha_s(y)\Delta_s$. 
\end{remark}

\begin{lemma}
The $\cha$-modules $\fM_{t,c}(\triv)$ and $M_{t,c}^\vee(\triv) = \ghom^\bullet_{\bbc[\fh]\rtimes\bbc G}(\cha,\bbc)$ are isomorphic.
\end{lemma}

\begin{proof}
In this proof, we shall use that $\cha = \bbc[\fh]\otimes\bbc G\otimes \bbc[\fh^*]$ and that the action of $\cha$ on $M_{t,c}^\vee(\triv)$ is given by $h\cdot\varphi(p\otimes a \otimes q) = \varphi((p\otimes a \otimes q)h)$, for any $\varphi\in\Hom_{\bbc[\fh]\rtimes\bbc G}(\cha,\bbc)$ and $h\in\cha$. Note that $p\otimes a\in \bbc[\fh]\rtimes\bbc G$ acts on $\triv$ by the constant coefficient of $p$. All that said, define $J:\fM_{t,c}(\triv)\to M_{t,c}^\vee(\triv)$ by
\[J(p_0)(p\otimes a \otimes q) = p(0)\lpi p_0, q\rpi,\]
for all $p\in\bbc[\fh],q\in\bbc[\fh^*]$ and $a\in \bbc G$. Here, $\lpi\cdot,\cdot\rpi$ is the polynomial pairing of (\ref{e:polyduality}). Note that $J(p_0)$ is indeed $\bbc[\fh]\rtimes\bbc G$-linear. Moreover, if $p_0$ is homogeneous of degree $d$, then $J(p_0)\in \ghom^d_{\bbc[\fh]\rtimes\bbc G}(\cha,\bbc)\cong \Hom_\bbc(\bbc[\fh^*]_d,\bbc)\cong\bbc[\fh]_d$, where $\cong$ here means linearly isomorphic. Therefore, $J$ is a well-defined map that induces bijections in each degree. All that is left to show is that $J$ intertwines the $\cha$-action. It is immediate to check that $J(Q(y)p_0) = y\cdot(J(p_0))$. Also,
\[(s\cdot J(p_0))(p\otimes a\otimes q) = J(p_0)((p\otimes a\otimes q)s) = p(0)\lpi p_0,s^{-1}(q)\rpi = J(Q(s)p_0)(p\otimes a\otimes q).\]
And finally, we have
\begin{align*}
(x\cdot J(p_0))(p\otimes a \otimes q) &= J(p_0)(p\otimes a \otimes [q,x])\\
&= J(p_0)(p\otimes a \otimes t\partial_x(q)) - \sum_s\frac{c(s)}{d^\vee_s}x(\alpha_s^\vee)J(p_0)(p\otimes a \otimes\Delta^\vee_s(q)s)\\
&=p(0)\lpi t\mu_x(p_0), q\rpi - \sum_s\frac{c(s)}{d_s^\vee}x(\alpha_s^\vee)p(0)\lpi I_{s^{-1}}(s(p_0)),q\rpi.
\end{align*}
Since $\frac{1}{d_s^\vee}I_{s^{-1}}\circ s = \frac{1}{d_s}I_s$, it follows that the last expression equals $J(Q(x)p_0)(p\otimes a \otimes q)$, as required.
\end{proof}

\begin{corollary}
If $\tau\in\Rep(G)$, the $\cha$-module $(Q,\fM_{t,c}(\tau))$ is in $\Co_{t,c}(G,\fh)$. Moreover, it is isomorphic to the costandard module $M^\vee_{t,c}(\tau)$.
\end{corollary}

\begin{proof}
Twisting with $\tau\in\Rep(G)$ the isomorphism of the previous lemma, we get $\fM_{t,c}(\tau)\cong M^\vee_{t,c}(\tau)$.
\end{proof}

In particular, $\fM_{t,c}(\tau)$ has a unique irreducible submodule, namely, $L_{t,c}(\tau)$.

\begin{definition}\label{d:M-singvec} 
A nonzero vector in $\fM_{t,c}(\tau)$ is called \emph{($\fh$-)singular} if $Q(y)v=0$ for all $y\in \fh$.
\end{definition}

\begin{proposition}\label{p:uniqueSubMod}
Suppose $\tau\in\Irr(G)$. Then, $Q(\cha)(1\otimes \tau) \cong L_{t,c}(\tau)$. \end{proposition}

\begin{proof}
From \cite[Lemma 2.12]{GGOR}, an $\cha$-module $M$ is isomorphic to $L_{t,c}(\tau)$ with $\tau$ simple if and only if $M = \cha\cdot(p(M))$ and $p(M)\cong\tau$, where $p(M)$ is the set of singular vectors of $M$, that is, those annihilated by the elements of positive degree in $\bbc[\fh^*]$. That is the case for $M = Q(\cha)\cdot(1\otimes\tau)$.
\end{proof}

\begin{lemma}\label{l:zeta}
Suppose $\epsilon$ is a character of the complex reflection group $G$ and let $\zeta_{c,\epsilon} = \sum_s \tfrac{c(s)}{d_s}\epsilon(s)(1-s)$. Then, $\zeta_{c,\epsilon}$ acts on the reflection representation $\fh$ by a scalar $h_{c,\epsilon}$, and in its dual $\fh^*$ by $h^*_{c,\epsilon}$. These scalars satisfy
\begin{equation}\label{e:zeta}
\begin{aligned}
h_{c,\epsilon} &= \frac{2}{r}\sum_s-c(s)(\epsilon\otimes\det\nolimits_\fh)(s)\\
h^*_{c,\epsilon} &= \frac{2}{r}\sum_sc(s)\epsilon(s)
\end{aligned}
\end{equation}
\end{lemma}

\begin{proof}
Using that $\zeta_{c,\epsilon}$ is in the centre of $\bbc G$, it is straightforward, by computing the traces.
\end{proof}

\begin{remark}
When $G$ is a real reflection group, we have that $\fh\cong\fh^*$ and both scalars above agree.
\end{remark}

\begin{proposition}\label{p:gen-one-dimensional}
The simple $\cha$-module $L_{t,c}(\tau)$ is one dimensional if and only if $\tau = \epsilon$ for a character $\epsilon$ and $(t - h^*_{c,\epsilon}) = 0$, with $h_{c,\epsilon}^*$ as in Lemma \ref{l:zeta}. 
\end{proposition}

\begin{proof}
Certainly, $\dim\tau=1$ if $\dim L_{t,c}(\tau) = 1$. From Proposition \ref{p:uniqueSubMod}, if $z_\epsilon$ is a fixed nonzero vector in $\bbc_\epsilon$, we have that $L_{t,c}(\epsilon) \cong Q(\cha)\cdot(1\otimes z_\epsilon)$. On the one hand, $Q(y)(1\otimes z_\epsilon) = 0$ for all $y\in\fh$ and $Q(s)(1\otimes z_\epsilon)\in\bbc(1\otimes z_\epsilon)$, for all $s\in \Cs$. On the other, for any $x\in\fh^*$, we have
\begin{align*}
Q(x)(1\otimes z_\epsilon) &= tx\otimes  z_\epsilon\ - \sum_s\frac{c(s)}{d_s}x(\alpha_s^\vee)d_s\alpha_s\otimes \epsilon(s)z_\epsilon\\
&=x\left(t - \zeta_{c,\epsilon}(x)\right)\otimes 1 \otimes z_\epsilon.
\end{align*}
The result follows from Lemma \ref{l:zeta}.
\end{proof}

\begin{corollary}\label{c:gen-one-dimensional2}
Suppose $G$ is a real reflection group and $c\in \bbc$. Then, $L_{t,c}(\epsilon)$ is one-dimensional if and only $c = \epsilon(s_0)\frac{t}{h}$, where $h$ is the Coxeter number of $G$ and $s_0$ is any reflection in $G$.
\end{corollary}

\begin{proof}
In this case, $h_{c,\epsilon} = (2|\Cs|/r)c\epsilon(s_0) = hc\epsilon(s_0)$.
\end{proof}

\subsection{The module $\fX_{t,c}(\tau)$} We now construct an integral-reflection representation dual to the projective module $X_{t,\check c}(\tau^*)$.

\begin{theorem}\label{t:X-module}
The linear space $\fX_{t,c}(\tau)=\bbc[\fh]\otimes\bbc[\fh^*] \otimes \tau$ endowed with operators $Q(x), Q(y)$ and $Q(s)$, for $x\in \fh^*$, $y\in\fh$ and $s\in \Cs$ defined by
\begin{align*}
Q(x)(p\otimes q\otimes z)&= p\otimes \partial_x(q)\otimes z +  t\mu_x(p)\otimes q\otimes z - \sum_s\frac{c(s)}{d_s}x(\alpha^\vee_s)I_s(p)\otimes s(q)\otimes s(z)\\
Q(y)(p\otimes q\otimes z)&= \partial_y(p)\otimes q\otimes z\\
Q(s)(p\otimes q\otimes z)&= s(p)\otimes s(q) \otimes s(z)
\end{align*}
extends to a module for the rational Cherednik algebra $\cha$. As before, $d_s = \frac{1-\lambda_s}{2}$.
\end{theorem}
\begin{proof}
The proof of this theorem is similar to the $(Q,\fM_{t,c}(\tau))$-case.
\end{proof}

We note that $\fX_{t,c}(\tau)$ is not in $\Co_{t,c}(G,\fh)$. Now, for any $n\in\bbz_{\geq 0}$ denote by $\bbc[\fh^*]_{\leq n}$ the subspace of $\bbc[\fh^*]$ consisting of polynomials of degree at most $n$ and $\bbc[\fh^*]_{n}$ the subspace of homogeneous degree $n$ polynomials. Note that $\bbc[\fh^*]_n$ is a $\bbc G$-module for each $n\in \bbz_{\geq 0}$.

\begin{definition}\label{d:filtration}
Let $\tau\in\Rep(G)$. Put $F_{-1}(\fX_{t,c}(\tau)) = 0$ and for each $n\in\bbz_{\geq 0}$, define 
\begin{equation*}
F_n(\fX_{t,c}(\tau)) = \bbc[\fh]\otimes\bbc[\fh^*]_{\leq n}\otimes\tau.
\end{equation*}
This defines an ascending filtration of $\fX_{t,c}(\tau)$ of $\bbc$-vector spaces
\[0\subseteq F_0(\fX_{t,c}(\tau))\subseteq F_1(\fX_{t,c}(\tau))\subseteq F_2(\fX_{t,c}(\tau))\subseteq \cdots\]
with $\cup_{n\in\bbz_{\geq 0}} F_n(\fX_{t,c}(\tau)) = \fX_{t,c}(\tau)$.
\end{definition}

\begin{proposition}\label{p:filtration}
Let $\tau$ be a finite-dimensional representation of $G$ and $n\in\bbz_{\geq 0}$. Then:
\begin{enumerate}
\item $F_n(\fX_{t,c}(\tau))$ is an $\cha$-module.
\item $F_0(\fX_{t,c}(\tau))\cong \fM_{t,c}(\tau)$ is in $\Co_{t,c}(G,\fh)$.
\item $F_n(\fX_{t,c}(\tau))/F_{n-1}(\fX_{t,c}(\tau))\cong \fM_{t,c}(\bbc[\fh^*]_n\otimes\tau)$ is in $\Co_{t,c}(G,\fh)$.
\item $F_n(\fX_{t,c}(\tau))$ is in $\Co_{t,c}(G,\fh)$.
\end{enumerate}
\end{proposition}

\begin{proof}
Item (1) follows directly from the expressions of the $Q$-operators defined in Theorem \ref{t:X-module}. For (2), note that
\[Q(x)(p\otimes 1 \otimes z) = t\mu_x(p)\otimes 1 \otimes z - \sum_s\frac{c(s)}{d_s}x(\alpha_s^\vee)I_s(p)\otimes 1 \otimes s(z),\]
which coincides, upon identifying $\bbc[\fh]\otimes \bbc\otimes \tau =\bbc[\fh]\otimes \tau$, with the $\cha$-action on $\fM_{t,c}(\tau)$ as in Theorem \ref{t:M-module}. The proof of (3) is similar. For (4), assuming by induction that $F_{n-1}(\fX_{t,c}(\tau))\in\Co_{t,c}(G,\fh)$, from the exact sequence
\[0\to F_{n-1}(\fX_{t,c}(\tau))\to F_{n}(\fX_{t,c}(\tau))\to \fM_{t,c}(\bbc[\fh^*]_n\otimes\tau)\to 0,\]
the result follows as $\Co_{t,c}(G,\fh)$ is closed for extensions.
\end{proof}

\begin{definition}\label{d:X-singvec} 
A nonzero vector in $\fX_{t,c}(\tau)$ is called \emph{($\fh$-)singular} if $Q(y)v=0$ for all $y\in \fh$.
\end{definition}

\begin{lemma}\label{l:h-singular}
The $\fh$-singular vectors of $\fX_{t,c}(\tau)$ are the nonzero vectors in $\bbc\otimes \bbc[\fh^*]\otimes\tau$. Let $\sigma$ be an irreducible $G$-representation occurring in $\bbc\otimes \bbc[\fh^*]_d\otimes\tau$. Then $Q(\cha)\cdot \sigma\cong L_{t,c}(\sigma)$.
\end{lemma}

\begin{proof}
The first statement is obvious since $Q(y)=\partial_y\otimes 1\otimes 1$. The second part follows again from \cite[Lemma 2.12]{GGOR}.
\end{proof}

\section{Dirac operators: global setting}\label{s:GlobalDirac}
\subsection{The global Dirac operator} Let $\tau$ be any finite-dimensional representation of $G$ and $\tau^*$ the contragredient representation. Start with the module $X_{t,\check c}(\widetilde\tau^*)$ for $\widetilde\tau^* = \tau^*\otimes S^*$. The Dirac element $\Cd$ yield an operator $\check D_\tau$ on $X_{t,\check{c}}(\widetilde\tau^*)$ by
\begin{equation}\label{e:dualDirac}
\check D_\tau(q\otimes p \otimes z^*\otimes \omega^*) = \sum_{j}\big(q\otimes p x_j\otimes z^*\otimes \fs^*(y_j)(\omega^*) + q\otimes p y_j\otimes z^*\otimes \fs^*(x_j)(\omega^*) \big),
\end{equation}
for all $q\in\bbc[\fh^*],p\in\bbc[\fh],z^*\in \tau^*$ and $\omega^*\in S^*$. Because the action of $x_j$ and $y_j$ on the $\chad$-part is by right multiplication, it is clear that the natural action of $\chad$ by left multiplication commutes with the operator $\check{D}_\tau$. 

\begin{definition} 
Let $\tau\in \Rep(G)$ and put $\widetilde\tau= \tau\otimes S\in\Rep(G)$, where $S=\bigwedge\fh$ is the spin module of $\Cc(\fh\oplus\fh^*)$. The \emph{global Dirac operator} $D_\tau:\fX_{t,c}(\widetilde\tau)\to\fX_{t,c}(\widetilde\tau)$ is defined by the equation
\[\lpi \phi, D_\tau(\xi) \rpi = \lpi \check D_\tau(\phi), \xi \rpi, \]
for all $\xi\in \fX_{t,c}(\widetilde\tau)$ and $\phi\in X_{\check c}(\widetilde\tau^*)$. Here, $\lpi\cdot,\cdot\rpi$ is the pairing of (\ref{e:Xduality}).
\end{definition}

It follows from the definition that $D_\tau$ commutes with the $Q$-action of $\cha$ on $\fX_{t,c}(\widetilde\tau)$. Our next task is to compute an explicit formula for the action of this operator. In order to do so, for each $1\leq j \leq r$, define the elements

\begin{equation}\label{e:deltas}
\delta_j= \partial_{y_j}\otimes 1\otimes 1,\qquad\textup{and}\qquad\delta^\vee_j= 1\otimes \partial_{x_j}\otimes 1
\end{equation}
of $\End(\bbc[\fh]\otimes\bbc[\fh^*]\otimes\tau)$.  We also make the following:

\begin{definition}\label{d:dualDunklops}
For each $y\in \fh$, define the operator $T^\vee_y(\tau):\bbc[\fh^*]\otimes\tau\to\bbc[\fh^*]\otimes\tau$ by 
\begin{equation}\label{e:dual-dunkl}
T^\vee_{y}(\tau)=t \mu_{y}\otimes 1+\sum_{s}\frac{c(s)}{d_s} \alpha_s(y) I_s^\vee\otimes s.
\end{equation} 
\end{definition}

 For  $1\leq j\leq r$, we put $T_j^\vee(\tau) = T_{y_j}^\vee(\tau)$. We shall view $T_j^\vee(\tau)\in\End(\bbc[\fh]\otimes \bbc[\fh^*]\otimes\tau)$ by letting it act trivially on $\bbc[\fh]$.

\begin{lemma}
The operators $\{T^\vee_{j}(\tau)\}$ commute.
\end{lemma}

\begin{proof}
The proof is completely analogous to that of Proposition \ref{p:x-Commutation}.
\end{proof}

\begin{remark}
We may refer to the operators $T_{y}^\vee(\tau)$ defined in (\ref{e:dual-dunkl}) as \emph{generalised dual Dunkl operators}. When $t=1$ and $\tau=\triv$, they are the dual (with respect to the pairing between $\bbc[\fh]$ and $\bbc[\fh^*]$) of the Dunkl operators for complex reflection groups studied in \cite{DJO} and \cite{DO}.
\end{remark}

\begin{proposition}
Let $\tau\in\Rep(G)$. The operator $D_\tau:\fX_{t,c}(\widetilde\tau)\to \fX_{t,c}(\widetilde\tau)$ satisfies the equation
\begin{equation}\label{e:DiracOp2}
\begin{aligned}
-D_\tau= \sum_j \delta^\vee_j \otimes \fs(y_j)+ \delta_j \otimes \fs(x_j) + T_j^\vee(\tau)\otimes\fs(x_j).
\end{aligned}
\end{equation}
\end{proposition}
\begin{proof}
We can assume that $\xi = p_0\otimes q_0 \otimes z \otimes \omega$, for $p_0\in\bbc[\fh],q_0\in\bbc[\fh^*],z\in\tau$ and $\omega\in S$. Now, for any $\phi = q\otimes p\otimes z^*\otimes \omega^*$ with $q\in\bbc[\fh^*],p\in\bbc[\fh],z^*\in\tau^*$ and $\omega^*\in S^*$, we have, 
\begin{align}\nonumber
\lpi \phi,D_\tau(\xi) \rpi &= \lpi \check D_\tau(\phi),\xi \rpi\\\label{eq:DOPinterm}
&= \sum_{j}\lpi q\otimes px_j\otimes z^*\otimes \fs^*(y_j)(\omega^*) + q\otimes p y_j\otimes z^*\otimes \fs^*(x_j)(\omega^*) ,\xi\rpi.
\end{align}
From (\ref{e:S-Cliffd}), we have $q\otimes p x_j\otimes z^*\otimes \fs^*(y_j)(\omega^*) = -q\otimes \mu_{x_j}(p)\otimes z^*\otimes \partial_{y_j}(\omega^*)$
which when dualised yield 
$\lpi q\otimes p x_j\otimes z^*\otimes \fs^*(y_j)(\omega^*),\xi\rpi =\lpi\phi, -p_0\otimes \partial_{x_j}(q_0)\otimes z\otimes \fs(y_j)(\omega)\rpi.$ 
This takes care of the first term of (\ref{eq:DOPinterm}). For the second term, we have

\begin{align*}
\sum_j q\otimes p y_j\otimes z^*\otimes \fs^*(x_j)(\omega^*)
&=2\sum_j\left\{\mu_{y_j}(q)\otimes p\otimes z^*\otimes \mu_{x_j}(\omega^*) + q\otimes [p,y_j]\otimes z^*\otimes \mu_{x_j}(\omega^*)\right\}. 
\end{align*}
Expanding the bracket $[p,y_j]$ in $\chad$ (using Proposition \ref{p:commutation2}), dualising,
exchanging $s$ by $s^{-1}$ in the sum, using $\alpha_s = \sum_j\alpha_s(y_j)x_j$ and observing that $d_{s^{-1}} = \frac{1-\lambda_{s^{-1}}}{2} = d_s^\vee$, we obtain the equation
\begin{multline}\label{e:DiracOp}
-D_\tau = \sum_j \left\{1\otimes \partial_{x_j}\otimes 1 \otimes \fs(y_j) + \partial_{y_j}\otimes 1\otimes 1 \otimes \fs(x_j)  +1\otimes t\mu_{y_j}\otimes 1 \otimes \fs(x_j)\right\} \\- \sum_s\frac{c(s)}{d^\vee_s}1\otimes I^\vee_s\otimes s\otimes \fs(\alpha_s)\circ s.
\end{multline}
Since $s(\omega)=\fs(\tau_s^\vee)(\omega)$, we have $\fs(\alpha_x)\circ s = \fs(\alpha_s\tau_s^\vee)$. From the straightforward equation $\frac 1{d_s^\vee}\alpha_s \tau_s^\vee=-\frac 1{d_s}\alpha_s$ (in  $\Cc$), we can rewrite (\ref{e:DiracOp}) as
\begin{equation}\label{e:DiracOp1}
\begin{aligned}
-D_\tau= \sum_j \left(1\otimes \partial_{x_j}\otimes 1 \otimes \fs(y_j)+ (\partial_{y_j}\otimes 1 + t\otimes \mu_{y_j})\otimes 1 \otimes \fs(x_j) \right) + \sum_s\frac{c(s)}{d_s}~1\otimes I^\vee_s\otimes s\otimes \fs(\alpha_s).
\end{aligned}
\end{equation}
From (\ref{e:DiracOp1}) and, again, the fact that $\alpha_s = \sum_j\alpha_s(y_j)x_j$, we obtain the desired equation.
\end{proof}

%
%
%
%
%

\subsection{The square of $D_\tau$}  Define the following elements:
\begin{equation}\label{e:nabla}
\begin{aligned}
\deg_{\fh^*} &= \sum_j\mu_{y_j}\partial_{x_j}&(\textup{in } \End(\bbc[\fh^*]))\\
\nabla &= \sum_j\partial_{y_j}\otimes\partial_{x_j}\otimes 1&(\textup{in } \End(\bbc[\fh]\otimes\bbc[\fh^*]\otimes\tau)).
\end{aligned}
\end{equation}
With respect to (\ref{e:deltas}), we have $\nabla = \sum_j\delta_j^\vee\delta_j=\sum_j\delta_j\delta_j^\vee$. We shall use equation (\ref{e:DiracOp2}) to compute $D_\tau^2$.  A large number of cancellations occur:
\begin{lemma}\label{l:somecancel}
The following identities hold in $\End(\fX(\widetilde\tau))$:
\[
\displaystyle
\begin{array}{rcr}
\displaystyle\sum_{i,j}\delta^\vee_i\delta^\vee_j\otimes \fs(y_iy_j) =0, &\qquad\qquad&  \displaystyle\sum_{i,j}\delta_i\delta_j\otimes \fs(x_ix_j) = 0,\\
\displaystyle\sum_{i,j}T_i^\vee(\tau)T^\vee_j(\tau)\otimes \fs(x_ix_j) = 0, &\qquad\qquad& \displaystyle\sum_{i,j}T_i^\vee(\tau)\delta_j\otimes \fs(x_ix_j) = 0.
\end{array}
\]

\begin{proof}
Clear, since in each case the part that acts on $\bbc[\fh]\otimes\bbc[\fh^*]\otimes\tau$ commutes while the Clifford algebra part anti-commutes.
\end{proof}

\end{lemma}
\noindent For the next identities, given $\tau\in\Rep(G)$, if convenient we shall omit from the notation the tensor legs in which an operator of $\fX_{t,c}(\tau\otimes S)$ acts as the identity ({\it e.g.}, $\kappa = 1\otimes 1 \otimes 1 \otimes \kappa$, {\it etc.}, where $\kappa$ was defined in (\ref{e:tau})).
\begin{lemma}\label{l:idents}
Suppose that $\tau\in\Irr(G)$. The following identities hold in $\End(\fX_{t,c}(\tau\otimes S))$, where $r = \dim\fh$:
\begin{align}
\sum_{i,j}\delta_i\delta^\vee_j \otimes \fs(y_ix_j) + \delta^\vee_i\delta_j \otimes\fs(x_iy_j) &= -2\nabla\label{e:sqD1}\\
\sum_{i,j}1\otimes \partial_{x_i}\mu_{y_j}\otimes 1 \otimes \fs(y_ix_j) + 1\otimes \mu_{y_i}\partial_{x_j}\otimes 1 \otimes \fs(x_iy_j) &= -2\deg_{\fh^*} - r -  \kappa\label{e:sqD2}\\
\sum_{i}1\otimes\Big(\partial_{x_i}I_s^\vee\otimes s\otimes \fs(y_i\alpha_s) + I_s^\vee \partial_{x_i}\otimes s\otimes \fs(\alpha_s y_i)\Big) &= 2\otimes  (s\otimes s\otimes s - 1\otimes s\otimes 1).\label{e:sqD3}
\end{align}
\end{lemma}

\begin{proof}
For (\ref{e:sqD1}), we break the sum $\sum_{i,j} = \sum_{i=j} + \sum_{i\neq j}$ and write
\begin{align*}
\sum_{i,j} 
&= \sum_{i\neq j} + \sum_k\partial_{y_k}\otimes \partial_{x_k}\otimes 1 \otimes \fs(y_kx_k + 1) + \partial_{y_k}\otimes \partial_{x_k}\otimes 1 \otimes \fs(x_ky_k+1) - 2\nabla\\
&= -2\nabla,
\end{align*}
as $x_jy_i = -y_ix_j$ if $i\neq j$, causing $\sum_{i\neq j} = 0$. For (\ref{e:sqD2}), using again $\sum_{i,j} = \sum_{i=j} + \sum_{i\neq j}$ we obtain, since $\partial_{x_k}\mu_{y_k}+\mu_{y_k}\partial_{x_k} = 2\mu_{y_k}\partial_{x_k}+1$ in the Weyl algebra, that
\begin{align*}
\sum_{i,j} &= \sum_{i\neq j} + \sum_k1\otimes \partial_{x_k}\mu_{y_k}\otimes 1 \otimes \fs(y_kx_k + 1) + 1\otimes \mu_{y_k}\partial_{x_k}\otimes 1 \otimes \fs(x_ky_k + 1) - 2\deg_{\fh^*} -r\\
&= - 2\deg_{\fh^*} - r - \sum_k1\otimes [\partial_{x_k},\mu_{y_k}]\otimes 1 \otimes \fs(x_ky_k + 1) + \sum_{i\neq j} 1\otimes [\partial_{x_i},\mu_{y_j}]\otimes 1 \otimes \fs(y_ix_j),
\end{align*}
and thus the claim. Finally, for the last equation, we get that the left hand side equals:
\begin{align*}
& \sum_i[\partial_{x_i}, I_s^\vee]\otimes s\otimes \fs(y_i\alpha_s)-2\sum_i \alpha_s(y_i) I_s^\vee \partial_{x_i}\otimes s\otimes 1,&\text{(using $\alpha_sy_i=-y_i\alpha_s-2\alpha_s(y_i)$ in $\Cc$)}\\
&=\sum_i d_s^\vee x_i(\alpha_s^\vee) s\otimes s\otimes \fs(y_i \alpha_s)-2 I_s^\vee\partial_{\alpha_s}\otimes s\otimes 1 &\text{(using Lemma \ref{l:intops})}\\
&= d_s^\vee s\otimes s\otimes \fs(\alpha_s^\vee \alpha_s) - 2(1-s) \otimes s\otimes 1 &\text{(again by Lemma \ref{l:intops})}\\
&=2(s\otimes s\otimes \tau_s - 1\otimes s\otimes 1),
\end{align*}
finishing the proof.
\end{proof}

\begin{corollary}\label{c:idents2}
 We have 
 \[\sum_{i,j}\Big(\delta_i^\vee T_j^\vee(\tau)\otimes\fs(y_ix_j) + T_i^\vee(\tau)\delta_j^\vee\otimes\fs(y_ix_j)\Big) = -2t(\deg_{\fh^*} + r/2 +  \kappa/2) - 2\otimes\sum_s\frac{c(s)}{d_s} (1\otimes s\otimes 1-s\otimes s\otimes s).\]
\end{corollary}

\begin{proof}
Just apply (\ref{e:sqD2}) and (\ref{e:sqD3}), in view of the definition of the dualised Dunkl operators in (\ref{e:dual-dunkl}).
\end{proof}

\begin{proposition}\label{p:D-tau-squared}
If $\tau\in\Irr(G)$, we have the following formula for the square of $D_\tau$ in $\End(\fX_{t,c}(\tau\otimes S))$:
\begin{equation}\label{e:D-tau-squared}
-\frac 12 D_\tau^2 = \nabla+t( \deg_{\fh^*}+ \kappa/2 + r/2) + N_c(\tau) - 1\otimes\sum_{s} \frac{c(s)}{d_s} s\otimes s \otimes s.
\end{equation}
\end{proposition}

\begin{proof}
When we compute $D_\tau^2$, using (\ref{e:DiracOp2}), the terms that survive the cancellations of Lemma \ref{l:somecancel} are
\begin{align*}
\sum_{i,j}\Big(\delta_i\delta^\vee_j\otimes \fs(y_ix_j) + \delta_i \delta_j^\vee \otimes \fs(x_iy_j) \Big)\\
\sum_{i,j}\Big(\delta_i^\vee T_j^\vee(\tau)\otimes\fs(y_ix_j) + T_i^\vee(\tau)\delta_j^\vee\otimes\fs(y_ix_j)\Big),
\end{align*}
which were treated by Lemma \ref{l:idents} and Corollary \ref{c:idents2}.
\end{proof}

\begin{corollary}
Let $\sigma,\tau\in\Irr(G)$. Suppose that $\sigma$ occurs in the subspace $\bbc\otimes\bbc[\fh^*]_d\otimes\tau\otimes {\bigwedge}^{\!\ell}\fh$ of $\fX_{t,c}(\tau\otimes S)$. Then, $D^2_\tau$ acts on this copy of $\sigma$ as the scalar
\[-\frac 12 D^2_\tau\mid_\sigma = t(d + \ell) + N_c(\tau) - N_c(\sigma).\]
In particular, $L_{t,c}(\sigma)\cong Q(\cha)\cdot \sigma$ is in the kernel of $D^2_\tau$ if and only if \[N_c(\sigma) - N_c(\tau) =t( d+\ell).\]
\end{corollary}

\begin{proof}
This is because on $\bbc\otimes\bbc[\fh^*]_d\otimes\tau\otimes {\bigwedge}^{\!\ell}\fh$, $\nabla$ acts by $0$, $\deg_{\fh^*}$ acts by $d$ on $\bbc[\fh^*]_d$, and $\kappa$ acts on ${\bigwedge}^{\!\ell}\fh$ by $2\ell-r$. For the second claim, recall that $L_c(\sigma)\cong Q(\cha)\cdot \sigma$ by Lemma \ref{l:h-singular}.
\end{proof}

\subsection{The kernel and cokernel of $D_\tau$} Since $D_\tau$ commutes with the $Q$-action of $\cha$, it is clear that both the kernel and the cokernel of $D_\tau$ are $\cha$-modules. In this subsection, we show that when $t\neq 0$, actually, they lie in category $\Co$. Before that, we state a corollary from the constructions made so far:

\begin{corollary}\label{c:lowdegree}
Let $z\neq 0$ be a fixed vector in $\tau$. Then, the operator $D_\tau\in\End_\bbc(\fX_{t,c}(\widetilde\tau))$ satisfy 
\[D_\tau(p\otimes q \otimes z \otimes 1) = -\sum_j p\otimes \partial_{x_j}(q) \otimes z \otimes y_j.\]
It follows that the kernel of $D_\tau$ contains a copy of $\fM_{t,c}(\tau)$ realised on the subspace $\bbc[\fh]\otimes \bbc \otimes \tau \otimes \bbc$.
\end{corollary}

\begin{proof}
The formula is clear from (\ref{e:DiracOp2}). As for the second statement, it follows from the description of the $Q$-operators action on the subspace $\bbc[\fh]\otimes \bbc \otimes \tau \otimes \bbc \cong \bbc[\fh]\otimes \tau$.
\end{proof}

We recall that $\cha\otimes\Cc$ is viewed as a $\bbz$-graded algebra graded by total degree, in which $\fh$ has degree $-1$, $\fh^*$ has degree $+1$ and $\bbc G$ has degree $0$ (note, in passing, that the elements $\tau_s^\vee$ and $\tau_s$ of (\ref{e:tau}) are, therefore, in degree $0$, which is consistent with the fact that they act as $s$ or $s^{-1}$ in the spin modules). We give the module $\fX(\widetilde\tau)$ a similar grading. Denote 
\[\fX(\widetilde\tau)_{m,d}^\ell=\bbc[\fh]_m\otimes \bbc[\fh^*]_d\otimes\tau\otimes{\bigwedge}^{\!\ell}\fh,
\]
and for each $n\in\bbz$ let $\fX(\widetilde\tau)_n = \oplus_{m-d-\ell=n} \fX(\widetilde\tau)_{m,d}^\ell$. This is the grading we shall be considering in this subsection. Two observations are important for what comes next: (i) the operator $D_\tau$ is homogeneous of degree $0$ and (ii) $D_\tau$ commutes with the $G$- action (by means of the $Q$-operators). It follows that $D_\tau$ is an operator on the isotypic components $\fX(\widetilde\tau)_{n,\sigma} = \fX_{t,c}(\widetilde\tau)_n\cap\fX_{t,c}(\widetilde\tau)_\sigma$, for each $\sigma\in\Irr(G)$ and $n\in\bbz$.

\begin{theorem}\label{t:fredholm}
Suppose $t\neq 0$. The kernel and the cokernel of $D^2_\tau$ are in $\Co_{t,c}(G,\fh)$.
\end{theorem}

\begin{proof}
We start with the kernel. Fix $\sigma\in\Irr(G)$. For $n,m\in \bbz$ write
\begin{equation}\label{e:catO1}
\fX(\widetilde\tau)_{n,\sigma}\{m\} = \bigoplus_{(d,\ell)\colon d+\ell = m-n} \fX(\widetilde\tau)_{m,d}^\ell.
\end{equation}
That is,  $\fX(\widetilde\tau)_{n,\sigma}\{m\}$ is the direct sum of all spaces in the degree-$n$ piece $\fX(\widetilde\tau)_{n,\sigma}$ with a fixed $m\in\bbz$, which we can assume to be in $\bbz_{\geq 0}$. Note that $D_\tau^2$ does not preserve $\fX(\widetilde\tau)_{n,\sigma}\{m\}$. However, from (\ref{e:D-tau-squared}) it satisfies the equation
\begin{equation}\label{e:catO2}
-\tfrac{1}{2}D^2_\tau|_{\fX(\widetilde\tau)_{n,\sigma}\{m\}} = \nabla + \mu(m)\textup{Id},
\end{equation}
where $\nabla$ was defined in (\ref{e:nabla}) and $\mu(m) = N_c(\tau)-N_c(\sigma) + t(m-n)$ is a scalar that only depends on $m$ (we recall that $d+\ell = m-n$ in the space we are considering). Now let $\xi = \sum_{m=0}^{m_0}\xi\{m\}$ be any element in $\fX(\widetilde\tau)_{n,\sigma}$, written in terms of the decomposition (\ref{e:catO1}). Assume that the element $\xi\{m_0\}$ in the biggest $m$-indexed part is nonzero. We then have
\begin{equation*}
-\tfrac{1}{2}D_\tau^2(\xi) = \mu(m_0)\xi\{m_0\} + \sum_{m=0}^{m_0-1}\Big(\nabla(\xi\{m+1\}) +\mu(m)\xi\{m\}\Big).
\end{equation*}
The assumption $\xi\in\ker (D_\tau^2)$ implies $\mu(m_0) = N_c(\tau)-N_c(\sigma) + t(m_0-n) = N_c(\tau)-N_c(\sigma) + t(d_0+\ell_0) = 0$. Since there are only finitely many possibilities for $d\in\bbz_{\geq 0}$ for which an equation
\begin{equation}\label{e:catO3}
N_c(\tau)-N_c(\sigma) + t(d+\ell) = 0
\end{equation}
is true, it follows that there exists $d_0\in\bbz_{\geq 0}$ such that $\xi\in F_{d_0}(\fX_{t,c}(\widetilde\tau))$ (see Definition \ref{d:filtration}). It follows that $\ker (D_\tau^2)$ is itself contained in $F_{d}(\fX_{t,c}(\widetilde\tau))$ for a sufficiently large $d$. This last module is in category $\Co$, by Proposition \ref{p:filtration}. This settles the claim for the kernel.

Now we consider the cokernel. Write $\coker (D_\tau^2) = \oplus_\sigma \Ck_\sigma$, where $\Ck_\sigma = (\fX_{t,c}(\widetilde\tau)_\sigma/(\fX_{t,c}(\widetilde\tau)_\sigma\cap \im D_\tau^2)$. We shall show that there exists a sufficiently large $d\in\bbz$ such that $F_d(\fX_{t,c}(\widetilde\tau)_\sigma)$ surjects onto $\Ck_\sigma$, when we take the canonical projection. Fix $n\in\bbz$ and $\sigma\in\Irr(G)$. Decompose $\fX(\widetilde\sigma)_{n,\sigma}$ as in (\ref{e:catO1}). We have two possibilities: either the scalar $\mu(m)$ of equation (\ref{e:catO2}) is nonzero, for all $m\in\bbz_{\geq 0}$, or there exists values of $m$ for which it becomes zero. The first case implies that $\fX(\widetilde\tau)_{n,\sigma}$ is actually in the image of $D_\tau^2$. Indeed, starting with $m=0$, we have $D_\tau^2|_{\fX(\widetilde\tau)_{n,\sigma}\{0\}} = \mu(0)\textup{Id}$ is invertible. Assuming, by induction, that $\fX(\widetilde\tau)_{n,\sigma}\{m\}$ is in the image, (\ref{e:catO2}) together with $\mu(m+1)\neq 0$ implies that $\fX(\widetilde\tau)_{n,\sigma}\{m+1\}$ is also in the image. It follows that $\fX(\widetilde\tau)_{n,\sigma}\to 0$ under the quotient map, in this case. On the other hand, assume that there exists values of $m$ for which $\mu(m) = 0$. Necessarily, there are only finitely many possible values of $m$ for which $\mu(m) = 0$. If $m_0$ is the largest possible one, the previous inductive argument can be applied to show that $\fX(\widetilde\tau)_{n,\sigma}\{m\}\subseteq \im(D_\tau^2)$ for $m\geq m_0$.

Just as in the case of the kernel, since $\mu(m) = N_c(\tau)-N_c(\sigma) + t(m-n) = N_c(\tau)-N_c(\sigma) + t(d+\ell)$, there are only finitely many possibilities for $d$ so that (\ref{e:catO3}) holds. Thus, for a sufficiently large $d\in\bbz_{\geq 0}$, we have that $F_d(\fX_{t,c}(\widetilde\tau)_\sigma)$ surjects onto $\Ck_\sigma$, and, since $\Irr(G)$ is finite, we are done.
\end{proof}

\begin{corollary}\label{c:fredholm}
Suppose $t\neq 0$. The kernel and the cokernel of $D_\tau$ are in $\Co_{t,c}(G,\fh)$.
\end{corollary}

\begin{remark}
When $t=0$, we may work with the restricted category $\Co$ for the finite-dimensional restricted algebra $\reschazero$. In that setting, there are analogous modules for the ones considered here and they are all finite dimensional.
\end{remark}

\begin{proof}
Clear since $\ker D_\tau$ is a submodule of $\ker D^2_\tau$ and $\coker D_\tau^2$ surjects onto $\coker D_\tau$.
\end{proof}

Similarly to the local case, let us write $\fX_{t,c}(\widetilde\tau)^\pm=\fX_{t,c}(\tau\otimes S^\pm).$ We may also define $D_\tau^+$, $D_\tau^-$ as the restrictions of $D_\tau$ to $\fX_{t,c}(\widetilde\tau)^+$ and $\fX_{t,c}(\widetilde\tau)^-$, respectively. To work with all three types of operators simultaneously, let us denote them by $D_\tau^\varepsilon$, where $\varepsilon$ could be $+$, $-$, or empty. Then, 
\[D_\tau^\varepsilon: \fX_{t,c}(\widetilde\tau)^\varepsilon\to\fX_{t,c}(\widetilde\tau)^{-\varepsilon}.\]
Note that  $\fX_{t,c}(\widetilde\tau)^\varepsilon$ is an element of the category $\Co^{\textup{ln}}_{t,c}(G,\fh)$ \cite[Section 2.2]{GGOR} of locally nilpotent left $\cha$-modules (not necessarily finitely generated). {\it A priori}, the kernel and the cokernel of $D_\tau$ are just in $\Co^{\textup{ln}}_{t,c}(G,\fh)$. Define the global indices 
\[I(\tau)^\pm=\ker D_\tau^\pm - \coker D_\tau^\pm.\] 
Since the Euler-Poincar\'e principle also holds in this category \cite[Chapter II, Proposition 6.6]{We}, we obtain, 
\[I^+ = \fX_{t,c}(\widetilde\tau)^+ - \fX_{t,c}(\widetilde\tau)^- = -I^-,\]
identities in Grothendieck group of $\Co^{\textup{ln}}_{t,c}(G,\fh)$. Therefore, the indices defined above agree, up to a sign.

\begin{definition} \label{d:globalindex}
The \emph{global Dirac index of} $\tau$ is defined as $I(\tau) = \ker D_\tau^+ - \coker D_\tau^+$, an element in the Grothendieck group of $\Co^{\textup{ln}}_{t,c}(G,\fh)$.
\end{definition}

\begin{proposition}\label{p:globindex}
As virtual modules in $\Co^{\textup{ln}}_{t,c}(G,\fh)$, we have, for every $\tau\in\Irr(G)$,
\[I(\tau) = \fM_{t,c}(\tau).\]
In particular, the index is in category $\Co$ for all parameters $t,c$.
\end{proposition}

\begin{proof}
Recall the filtration $0\subseteq F_0(\fX_{t,c}(\widetilde\tau)^\varepsilon)\subseteq F_1(\fX_{t,c}(\widetilde\tau)^\varepsilon)\subseteq \cdots$ of $\fX_{t,c}(\widetilde\tau)^\varepsilon$ (see Definition \ref{d:filtration}). Write
\[G_d(\fX_{t,c}(\widetilde\tau)^\varepsilon) = F_d(\fX_{t,c}(\widetilde\tau)^\varepsilon)/F_{d-1}(\fX_{t,c}(\widetilde\tau)^\varepsilon).\]
From Proposition \ref{p:filtration} we obtain the identity
\begin{equation}\label{e:d-filtered-piece}
F_d(\fX_{t,c}(\widetilde\tau)^\varepsilon) = \sum_{i=0}^d \fM_{t,c}(\bbc[\fh^*]_d\otimes\tau\otimes{\bigwedge}^{\!\varepsilon}\fh) 
\end{equation}
in the Grothendieck group of $\Co_{t,c}(G,\fh)$. Now, if $n^\varepsilon_{d,\tau,\sigma}=\dim\Hom_G(\sigma,\bbc[\fh^*]_d\otimes\tau\otimes{\bigwedge}^{\!\varepsilon}\fh)$ we also have
\begin{equation}\label{e:costd-decomp}
\fM_{t,c}\left(\bbc[\fh^*]_d\otimes\tau\otimes{\bigwedge}^{\!\varepsilon}\fh\right)  = \sum_\sigma n^\varepsilon_{d,\tau,\sigma}\fM_{t,c}(\sigma).
\end{equation}
Using (\ref{e:d-filtered-piece}) and (\ref{e:costd-decomp}), we get
\begin{align*}
F_d(\fX_{t,c}(\widetilde\tau)^+)-F_d(\fX_{t,c}(\widetilde\tau)^-) &= \sum^d_{i=0}\sum_\sigma(n_{i,\tau,\sigma}^+-n_{i,\tau,\sigma}^-)\fM_{t,c}(\sigma) .
\end{align*}
On the other hand, we know that the graded $G$-character of $\bbc[\fh^*]\otimes\tau\otimes({\bigwedge}^{\!+}\fh -{\bigwedge}^{\!-}\fh)$  satisfy
\begin{align*}
\tau &= \sum_{d\geq 0} \bq^{-d}\sum_\sigma\sum_\ell (-1)^\ell \dim\Hom_G\left((\bbc[\fh^*]_d\otimes\tau\otimes{\bigwedge}^{\!\ell}\fh\right)\sigma\\
&= \sum_{d\geq 0} \bq^{-d}\sum_\sigma(n^+_{d,\tau,\sigma}-n^-_{d,\tau,\sigma})\sigma.
\end{align*}
It follows that $F_d(\fX_{t,c}(\widetilde\tau)^+)-F_d(\fX_{t,c}(\widetilde\tau)^-) = \fM_{t,c}(\tau)$, for all $d$, and we are done.
\end{proof}

\subsection{The action of $D_\tau$ on singular vectors}\label{ss:globDiracinSingVecs} If we are interested in the occurrence of a simple module $L_c(\sigma)$ in $\ker D_\tau$, we need to determine which $\fh$-singular copies of $\sigma$ are killed by $D_\tau$. 
Recall from Lemma \ref{l:h-singular}, that  $\bigoplus_{d,\ell}\fX(\widetilde\tau)_{0,d}^\ell$ is the space of $\fh$-singular vectors. From (\ref{e:DiracOp2}), we see that on $\fX(\widetilde\tau)_{0,d}^\ell$, $D_\tau$ acts by
\begin{equation}\label{e:D-0}
D_\tau(0,d,\ell)= -\sum_j\delta^\vee_j\otimes\fs(y_j) - \sum_j  T^\vee_{j}(\tau)\otimes \fs(x_j).
\end{equation}
Notice that $D_\tau$ preserves the space of $\fh$-singular vectors. Define:
\begin{equation}\label{e:delta-complex}
0\to \bbc\to \bbc[\fh^*]\xrightarrow{\Delta}\bbc[\fh^*]\otimes\fh\xrightarrow{\Delta}\bbc[\fh^*]\otimes{\bigwedge}^{\!2}\fh\rightarrow\cdots,
\end{equation}
and
\begin{equation}\label{e:dunkl-complex}
0\leftarrow \tau\leftarrow \bbc[\fh^*]\otimes\tau\xleftarrow{\eta}\bbc[\fh^*]\otimes\tau\otimes \fh\xleftarrow{\eta}\bbc[\fh^*]\otimes\tau\otimes{\bigwedge}^{\!2}\fh\leftarrow\cdots,
\end{equation}
where
\[
\begin{aligned}
\Delta: \bbc[\fh^*]_d\otimes{\bigwedge}^{\!\ell}\fh\to \bbc[\fh^*]_{d-1}\otimes{\bigwedge}^{\!\ell+1}\fh,\quad \Delta&=\sum_j \partial_{x_j}\otimes \fs(y_j),\\
\eta:  \bbc[\fh^*]\otimes \tau\otimes {\bigwedge}^{\!\ell}\fh\to \bbc[\fh^*]\otimes \tau\otimes {\bigwedge}^{\!\ell-1}\fh,\quad \eta&=\sum_j  T^\vee_{y_j}(\tau)\otimes \fs(x_j).
\end{aligned}
\]
\begin{lemma}
Consider $v\in \fX(\widetilde\tau)_{0,d}^\ell$. Then $v\in \ker D_\tau$ if and only if $v\in\ker\Delta\cap\ker\eta$.
\end{lemma}

\begin{proof}Clear by comparing the degrees of these operators.
\end{proof}

Recall the Koszul resolution:
\begin{equation*}
\begin{aligned}
0\leftarrow \bbc\leftarrow \bbc[\fh^*]\xleftarrow{\bd}\bbc[\fh^*]\otimes \fh\xleftarrow{\bd}\bbc[\fh^*]\otimes{\bigwedge}^{\!2}\fh\leftarrow\cdots,\\
\bd(q\otimes y_{i_1}\wedge\dots\wedge y_{i_\ell})=\sum_{j=1}^\ell (-1)^{j-1} qy_{i_j}\otimes y_{i_1}\wedge\dots\wedge\hat y_{i_j}\wedge\dots\wedge y_{i_\ell}.
\end{aligned}
\end{equation*}
Notice that $\bd=-1/2\sum_{j}\mu_{y_j}\otimes \fs(x_j)$, in other words, when $t\neq 0$ the first term in $\eta$ is $-2t\bd$ and we can write $\eta = -2t\bd + \sum_s\tfrac{c(s)}{d_s}I_s^\vee\otimes s \otimes\fs(\alpha_s)$. 

\begin{lemma}
\begin{enumerate}
\item[(a)] $\Delta^2=0$ and $\eta^2=0$;
\item[(b)] $\Delta \bd+\bd\Delta=\textup{deg},$ where $\textup{deg}:\bbc[\fh^*]_d\otimes{\bigwedge}^{\!\ell}\fh\to \bbc[\fh^*]_d\otimes{\bigwedge}^{\!\ell}\fh$ is $\textup{deg}(q\otimes \omega)=d+\ell$. In particular, (\ref{e:delta-complex}) is a resolution and so 
\[\ker\Delta\mid_{\bbc[\fh^*]}=\bbc,\quad \ker\Delta\mid_{\bbc[\fh^*]_d\otimes{\bigwedge}^{\!\ell}\fh}=\Delta\left(\bbc[\fh^*]_{d+1}\otimes{\bigwedge}^{\!\ell-1}\fh\right),\ \ell\ge 1.
\]
\end{enumerate}
\end{lemma}

\begin{proof}
Formulae (a) are immediate since the terms in the left side of the tensor product defining $\Delta$ or $\eta$ commute while the ones in the right side anti-commute.
Part (b) is easy and well known from the classical Koszul resolution. 
\end{proof}

\begin{remark}
The complex $\eta$ (\ref{e:dunkl-complex}) is the dual and the generalisation (for arbitrary $\tau$) of the one defined in \cite[sections 2.2-2.3]{DO} for $t=1$ and $\tau=\triv$.
\end{remark}

\subsection{The dual Dunkl-Opdam complex} In this subsection, we assume that $t=1$. To emphasise the dependence on the parameter $c$, write $\eta(c)$ for the complex (\ref{e:dunkl-complex}):
\[0 \leftarrow\bbc[\fh^*]\otimes\tau\xleftarrow{\eta}\bbc[\fh^*]\otimes\tau\otimes \fh\xleftarrow{\eta}\bbc[\fh^*]\otimes\tau\otimes{\bigwedge}^{\!2}\fh\leftarrow\cdots,
\]
and $H_i(\eta(c))$ for its homology groups, $0\le i\le r$. 
In particular, $\eta(0)=-2\bd$ which has $H_i(\eta(0))=0$ for $i>0$ and $H_0(\eta(0))=\tau$. It is clear that in general, $\eta(c)(\bbc[\fh^*]\otimes\tau\otimes \fh)\subseteq \bbc[\fh^*]_{\ge 1}\otimes \tau$, hence 
\[H_0(\eta(c))\supseteq\tau.\]
 The following result is the analogue of \cite[Theorem 2.9 and Corollary 2.14]{DO}, with a very similar proof.

\begin{theorem}\label{t:dunkl-opdam}The following are equivalent:
\begin{itemize}
\item[(i)] $H_i(\eta(c))=0$ for all $i>0$;
\item[(ii)] $H_0(\eta(c))=\tau$.
\end{itemize}
\end{theorem}

\begin{proof}We follow closely the proof of \cite[Theorem 2.9]{DO}. That part (i) implies (ii) follows immediately from the graded Euler-Poincar\'e principle:
\[\sum_{i\ge 0}(-1)^i\textup{dimgr} ~H_i(\eta(c))=\sum_{i\ge 0}(-1)^i \textup{dimgr}~ \bbc[\fh^*]\otimes\tau\otimes{\bigwedge}^{\!i}\fh=\tau.
\]
For the other direction, denote for simplicity $K_m^\ell=\bbc[\fh^*]_m\otimes\tau\otimes {\bigwedge}^{\!\ell}\fh$, and we construct inductively a linear map
\[\Cv(c): K^\bullet\to K^\bullet,\text{ such that }\Cv(c)(K_m^\ell)\subseteq K_m^\ell,\text{ and which satisfies }
\]
\begin{enumerate}
\item $\Cv(c)$ is the identity operator on $K_0^0=\tau$;
\item $\Cv(c)(p\otimes \omega)=(\Cv(c)p)\otimes \omega$, for all $p\in \bbc[\fh^*]\otimes\tau$ and $\omega\in \bigwedge \fh$;
\item $\Cv(c)\eta(c)=\eta(0)\Cv(c)$.
\end{enumerate}

Firstly, let us assume that such $\Cv(c)$ has been constructed. We claim that $\Cv(c)$ is a linear isomorphism. Suppose not and let $m$ be the smallest nonnegative integer such that $\Cv(c)$ is not an isomorphism on $K_m^\bullet$. Because of (2), $\Cv(c)$ is not an isomorphism on $K_m^0$ then. By (1), we must have $m\ge 1$, and there exists $p\in K_m^0$ such that $p\notin \im \Cv(c)$. Because of the properties of the complex $\eta(0)$, there exists $q\in K_{m-1}^1$ such that
\[p=\eta(0)q.
\] 
By the assumption on $m$, $\Cv(c)$ must be an isomorphism on $K_{m-1}^1$, so there exists a unique $q'\in K_{m-1}^1$ such that $q=\Cv(c) q'$. Then $p=\eta(0)\Cv(c)q'=\Cv(c) \eta(c) q'$ by (3). But this is a contradiction since we have assumed that $p$ is not in the image of $\Cv(c)$.

Similarly, one can prove that $\Cv(c)$ is in fact unique. Since $g\cdot \Cv(c)\cdot g^{-1}$ for $g\in G$ also satisfies (1),(2),(3) when $\Cv(c)$ does (because $\eta(c)$ is $G$-invariant), we  see that in fact $\Cv(c)$ is also $G$-invariant.

Now, since $\Cv(c)$ gives an isomorphic intertwiner between $\eta(c)$ and $\eta(0)$, we have that the homology groups are the same, and in particular (i) holds.

It remains to construct $\Cv(c)$ inductively under the assumption that (ii) holds. Take $m\ge 1$ and suppose $\Cv(c)$ has been constructed on $K^\bullet_i$ for all $i<m$ (the base case is given by (1) and (2)). Let $p\in \bbc[\fh^*]_m\otimes\tau$ be given. By the assumption (ii), there exists $q\in \bbc[\fh^*]_{m-1}\otimes \tau\otimes \fh$ such that
\[p=\eta(c)q.
\]
By induction, $\Cv(c)\eta(c)q=\eta(0)\Cv(c)q$, where $\Cv(c)q$ is known, as $q\in K_{m-1}^1$. Hence, define $\Cv(c)p=\eta(0)\Cv(c)q$. 
The definition is completed by extending $\Cv(c)$ on $K^\bullet_m$ using the required property (2).
\end{proof}

\begin{remark}\label{r:DO}
\begin{enumerate}
\item Condition (ii) in Theorem \ref{t:dunkl-opdam} is equivalent with the requirement that
\[\sum_{j} \im T_{y_j}^\vee(\tau)=\bbc[\fh^*]_{\ge 1}\otimes \tau.
\]
\item It would be interesting to determine the precise parameters $c$ for which the homology of the complex is nontrivial, and in such cases, determine the homology groups. This would be a generalisation of the results obtained in the case $\tau=\triv$ in \cite{DJO,DO}.
\end{enumerate}
\end{remark}

\subsection{Particular cases}The two extreme cases when $\Delta=0$ on $\bbc[\fh^*]_d\otimes{\bigwedge}^{\!\ell}\fh$ are: (a) $\ell=r$, the top degree, and (b) $d=0$.

\medskip

Consider the particular case (a). Denote by $\vol=y_1\wedge\ldots\wedge y_r\in S$, the canonical generator of the top degree space in the spin module. Also, let
\[\omega_j = \partial_{x_j}(\vol) = (-1)^{j-1}y_1\wedge\ldots\wedge \hat y_j\wedge\ldots \wedge y_r\in S,\]
where, as usual, the symbol $\hat y_j$, means we omitted $y_j$.

\begin{proposition}\label{p:topdegree}
On $\fX(\widetilde\tau)_{0,d}^r$, we have:
\[-\frac{1}{2}D_\tau(0,d,r):\quad 1\otimes v(d,\tau) \otimes \vol \mapsto \sum_j 1\otimes T_{y_j}^\vee(\tau)(v(d,\tau))\otimes \omega_j,\]
where $v(d,\tau)\in \bbc[\fh^*]_d\otimes\tau$. Therefore, $1\otimes v(d,\tau)\otimes \vol\in \ker D_\tau$ if and only if 
\[v(d,\tau)\in\bigcap_j \ker  T_{y_j}^\vee(\tau).
\]
\end{proposition}
\begin{proof}
The expression of $D_\tau(0,d,r)$ is immediate from (\ref{e:D-0}). For the second part, notice that $\omega_j$ are linearly independent in ${\bigwedge}^{\!r-1}\fh.$
\end{proof}

\begin{example}\label{c:gen-one-dimensional3}
Let $\epsilon$ be a character of the complex reflection group $G$ and let $h_{c,\epsilon}$ be as in Lemma \ref{l:zeta}. If $t+h_{c,\epsilon} = 0$, then $L_{t,c}(\det_\fh\otimes\epsilon)$ is one dimensional and occurs in the kernel of $D_\epsilon$.
\end{example}

\begin{proof}
By the previous proposition and since $t+h_{c,\epsilon} = 0$, $D(1\otimes 1\otimes z_\epsilon\otimes\vol) = 0$. As $F_0(\fX_{t,c}(\epsilon\otimes{\bigwedge}^{\!r}\fh))$ is isomorphic to $\fM_{t,c}(\epsilon\otimes\det_\fh),$ the result follows from Proposition \ref{p:gen-one-dimensional}.
\end{proof}

Now consider the other extreme $d=0$. 

\begin{lemma}\label{l:d=0} 
On $\fX(\widetilde\tau)_{0,0}^\ell$, we have
\[D_\tau(0,0,\ell)= 1\otimes\sum_j \mu_{y_j}\otimes\left((t+N_c(\tau))(1\otimes \fs(x_j))-\sum_s\frac{c(s)}{d_s}s\otimes \fs(s^{-1}(x_j)) \right).
\]
\end{lemma}

\begin{proof}
It is easier to use (\ref{e:DiracOp1}). Since $I_s^\vee(1)=d_s^\vee \alpha_s^\vee$ and $\alpha_s^\vee=\sum_j x_j(\alpha_s^\vee) y_j$, we get \[D_\tau(0,0,\ell)= 1\otimes\sum_j \mu_{y_j}\otimes\left(t(1\otimes \fs(x_j))-\sum_s\frac{c(s)}{\lambda_s}s\otimes x_j(\alpha_s^\vee)\fs(\alpha_s) \right).
\]
Then, we write $\lambda_s^{-1}x_j(\alpha_s^\vee)\alpha_s=s^{-1}(x_j)-x_j$ and finally we use that $\sum_s\frac{c(s)}{d_s}s$ acts by $N_c(\tau)$ in $\tau$.
\end{proof}

For every $x\in\fh^*$, define
\begin{equation}\label{e:d=0}
t_x(\tau): \tau\otimes{\bigwedge}^{\!\ell}\fh\to \tau\otimes{\bigwedge}^{\!\ell-1}\fh,\quad t_x(\tau)=(t+N_c(\tau))(1\otimes \fs(x_j))-\sum_s\frac{c(s)}{d_s}s\otimes \fs(s^{-1}(x_j)) .
\end{equation}
Since $y_j$ are linear independent in $\fh$, we see that $D_\tau(0,0,\ell)=0$ if and only if $t_{x_j}(\tau)=0$ for all $j$. Hence:

\begin{proposition}\label{p:d=0}
The kernel of $D_\tau$ on $\fX(\widetilde\tau)_{0,0}^\ell$ consists of $\bbc\otimes\bbc\otimes u(\tau,\ell)$, where 
\[u(\tau,\ell)\in \bigcap_j \ker t_{x_j}(\tau)\subseteq\tau\otimes{\bigwedge}^{\!\ell}\fh.
\]
\end{proposition}

\subsection{Local-global relations} We relate the kernel of the global Dirac operator to that of the local Dirac operator. This is, of course, inspired by the ideas from the classical setting of Dirac operators on symmetric spaces \cite{AS}, \cite{Pa}, \cite{Vo}. 

Let $\widetilde \tau$ be a $G$-representation. We shall use repeatedly the duality
\[\cha^{\textup{op}}\cong \chad.
\]
We begin by observing that our integral-reflection module $\fX_{t,c}(\widetilde\tau)$ is an explicit realisation of the $(\bbz\times\bbz)$-graded dual of the projective module $X_{t,\check c}(\widetilde\tau^*) \cong \bbc[\fh^*]\otimes\bbc[\fh]\otimes\widetilde\tau^*$:
\begin{equation*}
\fX_{t,c}(\widetilde\tau)=\ghom^{\bbz\times\bbz}_{\bbc}(X_{t,\check c}(\widetilde\tau^*),\bbc).
\end{equation*}
In the previous identity, the $G$-module $\widetilde\tau^*\cong\bbc\otimes\bbc\otimes\widetilde\tau^*$ is in degree $(0,0)$. Note that, even though $X_{t,\check c}(\widetilde\tau^*)$ and $\fX_{t,c}(\widetilde\tau)$ are naturally $(\bbz\times\bbz)$-graded vector spaces, we can view them as $\bbz$-graded vector spaces by taking total degrees. We then have the following:

\begin{proposition}\label{p:Xish-finite}
The graded module $\fX_{t,c}(\widetilde\tau)$ equals the space of $\fh$-finite vectors of $\ghom^\bbz_\bbc(X_{t,\check{c}}(\widetilde\tau^*),\bbc)$.
\end{proposition}

\begin{proof}
We start by noting that $\Hom^{\bbz}_\bbc(X_{t,\check{c}}(\widetilde\tau^*),\bbc) \cong \oplus_{d\in\bbz} \Hom_\bbc(\oplus_{m-n=d}\bbc[\fh^*]_m\otimes\bbc[\fh]_n\otimes\widetilde\tau^*,\bbc)$ and so it can be identified with the space of all finite linear combinations $\sigma = \sum_d\sigma_d$ in which each summand is a formal power series $\sigma_d = \sum_j p_{j}\otimes q_{j}\otimes z_{j}$ ({\it i.e.}, not necessarily finitely many summands), such that $p_{j}\in\bbc[\fh],q_{j}\in\bbc[\fh^*],z_{j}\in\widetilde\tau$ and $\deg(p_{j})-\deg(q_{j}) = d$. Under this identification, the action of $\fh$ is by the $Q$-operators, and it is therefore by derivation of the polynomials $\{p_j\}$ in the direction of $y\in\fh$. Asking for $\fh$-finiteness, is equivalent to saying that each $\sigma_d$ is a polynomial, yielding the desired identification with $\ghom^{\bbz\times\bbz}_{\bbc}(X_{t,\check c}(\widetilde\tau^*),\bbc)=\fX_{t,c}(\widetilde\tau)$.
\end{proof}

For each ($\bbz$-)graded $\cha$-module $Y$ (respectively, $\chad$-module), denote the contragredient module
\begin{equation*}
Y^\dagger=\ghom^\bbz_{\bbc}(Y,\bbc),
\end{equation*}
which is a graded module for $\cha$ (respectively, $\chad$).
\begin{lemma}
If $Y$ is in $\Co_{t,c}(G,\fh)$, then $Y^\dagger$ is in $\Co_{t,\check c}(G,\fh^*)$. Also, $M_{t,\check{c}}(\tau^*)^\dagger\cong\fM_{t,c}(\tau)$ and 
\begin{equation*}
L_c(\tau)\cong L_{\check c}(\tau^*)^\dagger.
\end{equation*}
\end{lemma}

\begin{proof}
This is known. See \cite[section 4.2.1]{GGOR} or \cite[Proposition 3.32 and Proposition 3.34]{EM}).
\end{proof}

\begin{lemma}\label{l:duality}
Suppose $Y$ is a graded $\cha$-module. Then
\begin{equation}\label{e:hom-duality}
\ghom_{\cha}(Y,\fX_{t,c}(\widetilde\tau))=\ghom_G(\widetilde\tau^*,Y^\dagger\mid_G).
\end{equation}
\end{lemma}

\begin{proof}
Since $Y$ is in $\Co_{t,c}(G,\fh)$, it follows from the local $\fh$-nilpotency (which is equivalent to local $\fh$-finiteness, when $Y$ is graded) and from Proposition \ref{p:Xish-finite} that even though $\fX_{t,c}(\widetilde\tau)=\ghom^{\bbz\times\bbz}_{\bbc}(X_{t,\check c}(\widetilde\tau^*),\bbc)$, we actually have
\[\ghom^\bbz_{\cha}(Y,\fX_{t,c}(\widetilde\tau))=\ghom^\bbz_{\cha}(Y,\ghom^\bbz_{\bbc}(X_{t,\check c}(\widetilde\tau^*),\bbc)).\]
The rest of the proof is ``abstract nonsense":
\begin{equation}\label{e:nonsense}
\begin{aligned}
\ghom^\bbz_{\cha}(Y,\fX_{t,c}(\widetilde\tau))&=\ghom^\bbz_{\cha}(Y,\ghom^\bbz_{\bbc}(X_{t,\check c}(\widetilde\tau^*),\bbc))\\
&=\ghom^\bbz_{\chad}(X_{t,\check c}(\widetilde\tau^*),\ghom^\bbz_{\bbc}(Y,\bbc))\\
&=\ghom^\bbz_{\chad}(\chad\otimes_{\bbc G}\widetilde\tau^*,Y^\dagger)\\
&=\ghom^\bbz_G(\widetilde\tau^*,Y^\dagger\mid_G),
\end{aligned}
\end{equation}
where the last equality is Frobenius reciprocity.
\end{proof}
Let $\tau$ be an irreducible $G$-representation and specialise $\widetilde \tau=\tau\otimes S.$ Recall the global Dirac operators  $D^\varepsilon_\tau: \fX_{t,c}(\widetilde\tau)^\varepsilon\to \fX_{t,c}(\widetilde\tau)^{-\varepsilon}$, defined in the previous subsections. Each operator $D_\tau^\varepsilon$ commutes with the action of $\cha$, thus it induces linear maps
\begin{equation*}
D_\tau^\varepsilon(Y): \ghom_{\cha}(Y,\fX_{t,c}(\tau\otimes S^\varepsilon))\longrightarrow \ghom_{\cha}(Y,\fX_{t,c}(\tau\otimes S^{-\varepsilon})),
\end{equation*}
for every graded $\cha$-module $Y$.

Recall also the local Dirac operator $D_Y: Y\otimes S\to Y\otimes S$ for every $Y$ in $\Co_{t,c}(G,\fh)$, and similarly, $D_Y^\varepsilon$.  
The operator $D_Y^\varepsilon$ is $G$-invariant, and thus induces linear maps $\ghom_G(\sigma,Y\otimes S^\varepsilon)\to \ghom_G(\sigma,Y\otimes S^{-\varepsilon})$, which equivalently can be written as
\begin{equation*}
D_Y^\varepsilon(\sigma): \ghom_G(\sigma\otimes (S^\varepsilon)^*,Y)\longrightarrow \ghom_G(\sigma\otimes (S^{-\varepsilon})^*,Y).
\end{equation*}

\begin{proposition}\label{p:local-global} Let $Y$ be a module in $\Co_{t,c}(G,\fh)$. Denote by $\iota$ the map providing the identification of the two sides of (\ref{e:hom-duality}). The following diagram commutes:
\begin{equation}\label{e:local-global-diagram}
\begin{tikzcd}
\ghom_{\cha}(Y,\fX_{t,c}(\tau\otimes S^\varepsilon)) \xar{r}{D_\tau^\varepsilon(Y)} &\ghom_{\cha}(Y,\fX_{t,c}(\tau\otimes S^{-\varepsilon}))\\
\ghom_G(\tau^*\otimes (S^\varepsilon)^*,Y^\dagger)  \xar{r}{D_{Y^\dagger}^\varepsilon(\tau^*)}\xar{u}{\iota} & \ghom_G(\tau^*\otimes (S^{-\varepsilon})^*,Y^\dagger)\xar{u}{\iota}
\end{tikzcd}
\end{equation} 
In particular,
\begin{equation}\label{e:local-global-rels}
\begin{aligned}
\ghom_G(\tau^*,\ker D_{Y^\dagger}^\varepsilon)&\cong\ghom_{\cha}(Y,\ker D_\tau^\varepsilon),\\
\ghom_G(\tau^*,\coker D_{Y^\dagger}^\varepsilon)&\hookrightarrow\ghom_{\cha}(Y, \coker D_\tau^\varepsilon).\\
\end{aligned}
\end{equation}
\end{proposition}

\begin{proof}
The commutativity of the diagram follows at once by tracing the definitions of the Dirac operators in the chain of identifications (\ref{e:nonsense}). For the relations in (\ref{e:local-global-rels}), simply apply the respective hom-functors  to the corresponding exact sequence of kernel and cokernel to obtain the commutative diagram (for lack of space, we write $\Hom$ instead of $\ghom$ and omitted the indices) with exact rows:
\[
\begin{tikzcd}
\Hom(Y,\ker D_\tau^\varepsilon) \xar\mono{r}&\Hom(Y,\fX_{t,c}(\widetilde\tau)^\varepsilon) \xar{r}{D_\tau^\varepsilon(Y)} &\Hom(Y,\fX_{t,c}(\widetilde\tau)^{-\varepsilon})\xar{r}&\Hom(Y,\coker D_\tau^\varepsilon)\\
\Hom(\tau^*,\ker D_{Y^\dagger}^\varepsilon) \xar\mono{r}\xar[dotted]{u}{\varphi}&\Hom(\tau^*\otimes (S^\varepsilon)^*,Y^\dagger)  \xar{r}{D_{Y^\dagger}^\varepsilon(\tau^*)}\xar{u}{\iota} & \Hom(\tau^*\otimes (S^{-\varepsilon})^*,Y^\dagger)\xar\epi{r}\xar{u}{\iota}&\Hom(\tau^*,\coker D_{Y^\dagger}^\varepsilon).\xar[dotted]{u}{\psi}
\end{tikzcd}
\]
Here, the last row is exact since $\tau^*$ is projective in the category of (graded) $\bbc G$-modules. It is then clear that $\varphi$ is an isomorphism while $\psi$ is injective.
\end{proof}

\begin{remark}
The second relation of (\ref{e:local-global-rels}) is an isomorphism if and only if $\Hom(Y,\fX_{t,c}(\widetilde\tau)^{-\varepsilon})\to\Hom(Y,\coker D_\tau^\varepsilon)$ is surjective.
\end{remark}

%

\subsection{The $\bbz_2$-example}
For the particular case in which $G = \bbz_2$, and $t=1$ there are many simplifications to the formulae above. So suppose that $\dim\fh = 1$ and $G=\bbz_2$. Fix $x = \alpha$ and $y = \alpha^\vee/2$ basis of $\fh^*$ and $\fh$, respectively. The Cherednik relation is
\[[y,x] = 1 - 2cs.\]
More generally, Proposition \ref{p:commutation} reads here as
\begin{equation}\label{e:gensl2}
[y,x^m] = mx^{m-1} - c(1-(-1)^{m})x^{m-1}s.
\end{equation}
Fix $\tau\in\{\triv,\sgn\}$ and consider $\fX_{t,c}(\widetilde\tau) = \bbc[x]\otimes\bbc[y]\otimes\bbc_\tau\otimes\bigwedge\fh$. We shall denote by $\partial$ and $\mu$ the derivation and multiplication by $x$ or $y$ in the respective polynomial algebras. The formulae for $(Q,\fX_{t,c}(\widetilde\tau))$ and $D^\varepsilon_\tau$ become
\begin{equation*}
\begin{aligned}
Q(x)&=1\otimes\partial\otimes 1 \otimes 1 + \mu\otimes 1 \otimes 1 \otimes 1 - 2cI\otimes s\otimes s \otimes s\\
Q(y)&=\partial\otimes 1\otimes 1 \otimes 1 \\
D_\tau^+&= -1\otimes\partial\otimes 1 \otimes \fs(y)\\
D_\tau^-&= -\partial\otimes 1 \otimes 1 \otimes \fs(x) - 1\otimes(\mu\otimes 1 +  cI^\vee\otimes s)\otimes \fs(x),
\end{aligned}
\end{equation*}
where $\fs(y)(\cdot)=y\wedge(\cdot)$ and $\fs(x)(\cdot)=-2\partial(\cdot)$ in $\bigwedge\fh$. We now embark in the quest to determine, for $G = \bbz_2$, the kernel and the cokernel of $D_\tau^\varepsilon$.  Throughout, we shall make use of the following notation. For $m,n\in\bbz_{\geq 0}$ and a fixed $0\neq z_\tau\in\bbc_\tau$, define the elements
\begin{align*}
\psi_{m,d} &= x^m\otimes y^d\otimes z_\tau\otimes 1\in \fX(\widetilde\tau)_{m,d}^0\\
\xi_{m,d} &= x^m\otimes y^d\otimes z_\tau\otimes y\in \fX(\widetilde\tau)_{m,d}^1.
\end{align*}

\begin{lemma}\label{l:computation}
In terms of the basis $\{\psi_{m,d}\}$ of $\fX_{t,c}(\widetilde\tau)^{+}$ and the basis $\{\xi_{m,d}\}$ of $\fX_{t,c}(\widetilde\tau)^{-}$, we have:
\begin{align*}
\tfrac{1}{2}D_\tau^-(\xi_{m,2k-1}) &= m\psi_{m-1,2k-1} + \psi_{m,2k}\\
\tfrac{1}{2}D_\tau^-(\xi_{m,2k}) &= m\psi_{m-1,2k} + \left(\frac{2k+1 + 2c\tau(s)}{2k+1}\right)\psi_{m,2k+1}.
\end{align*}
\end{lemma}

\begin{proof}
Direct computation. Note that $I^\vee(y^d) = 0$ if $d$ is odd and it equals $\frac{2}{d+1}y^{d+1}$ if $d$ is even.
\end{proof}

\begin{corollary}\label{c:computation}
When $m=0$, we have
\[D_\tau^-(\xi_{0,d}) = 2\left(1 + c\tau(s)\frac{1+(-1)^{d}}{d+1}\right)1\otimes y^{d+1}\otimes z_\tau\otimes 1.\]
Thus, $D_\tau^-(\xi_{0,d}) = 0$ if and only if $d = 2k$ and $2c\tau(s) + (2k +1) = 0$.
\end{corollary}

\begin{proposition}\label{p:kernelD}
Fix $0\neq z_\tau\in\bbc_\tau$. Suppose that $2c\tau(s) + (2k+1) = 0$ with $k\in\bbz_{\geq 0}$.  Then, we realise the modules
\[Q(\cha)\cdot \xi_{0,2k} \cong L_c(\triv)\subseteq \fX_{t,c}(\sgn\otimes S),\]
if $2c = 2k + 1,\tau = \sgn$ and if $-2c = 2k+1,\tau = \triv$, 
\[Q(\cha)\cdot \xi_{0,2k} \cong L_c(\sgn)\subseteq \fX_{t,c}(\triv\otimes S)\]
in the kernel of $D_\tau^-$. These modules are of dimension $2k+1$.
\end{proposition}

\begin{proof}
First, for $c$ and $\tau$ arbitrary, from Lemma \ref{c:computation}, $D_\tau^-(\xi_{0,2k}) = 0$ if and only if $2c = 2k+1$ and $\tau = \sgn$ or $2c = -(2k+1)$ and $\tau = \triv$. Hence, the submodule of $\fX_{t,c}(\tau\otimes \bigwedge\fh)$ generated by $\xi_{0,2k}$ specified in the statement is in the kernel. We claim that for any $0\leq j \leq 2k$, we have $Q(x^j)(\xi_{0,2k})$ is nonzero. Indeed, note that
\[Q(x^j)(\xi_{0,2k}) = (a_01\otimes 1 + a_1x\otimes y + \cdots + a_j x^j\otimes y^j)y^{2k-j}\otimes z_\tau\otimes y,\]
for scalars $a_i\in\bbc$. Note that $a_0 = (2k)(2k-1)\cdots(2k-j+1)$, so $Q(x^j)(\xi_{0,2k})\neq 0$ for $j\leq 2k$. We thus obtain a polynomial expression for $Q(x^{2k})(\xi_{0,2k})$ with constant term equal to $(2k)!$. Applying $x$ once more we get
\[Q(x^{2k+1})(\xi_{0,2k}) = (b_1x\otimes y + \cdots + b_{2k} x^{2k+1}\otimes y^{2k+1})y^{-1}\otimes z_\tau\otimes y,\]
for scalars $b_j$. Assuming that $Q(x^{2k+1})(\xi_{0,2k})\neq 0$ implies that $Q(y)Q(x^{2k+1})(\xi_{0,2k})\neq 0$. Commuting and applying the Cherednik relation (\ref{e:gensl2}) gives
\[0\neq Q(y)Q(x^{2k+1})(\xi_{0,2k}) = Q([y,x^{2k+1}])(\xi_{0,2k}) = (2k+1 + 2c\tau(s))Q(x^{2k})(\xi_{0,2k}) = 0,\]
which is a contradiction.
\end{proof}
\begin{remark}
Note that $D_\tau^2$ acts on $\bbc\otimes \bbc[y]_m\otimes\bbc_\tau\otimes{\bigwedge}^{\!1}\fh$ as the scalar (cf. (\ref{e:D-tau-squared}))
\[-2(m + 1 + c\tau(s) - c\tau(s)(-1)^{m+1}),\]
which is zero if and only if $m=2k$ is even and $(2c\tau(s) + 2k+1) = 0$.
\end{remark}

In the rest of this section, we shall say that the parameter $c$ is \emph{regular} if $c\notin \frac{1}{2} + \bbz$. As is suggested by Lemma \ref{c:computation} and we shall shortly see, the kernel of $D_\tau^-$ will be nonzero only if $c$ is {\it not} regular. We now explicitly determine the kernel and cokernel of the operators $D_\tau^{\pm}$.

\begin{corollary}\label{c:kerZ2}
Let $G=\bbz_2$, $t=1$ and $\tau\in\{\triv,\sgn\}$. Then:
\[
\begin{array}{lcl}
\textup{ if }2c \in -\tau(s)(2\bbz_{\geq 0}+1), &\quad& \left\{
\begin{array}{l}
\ker D_\tau^+ \cong \fM_{1,c}(\tau)\\
\ker D_\tau^-\cong L_{1,c}(\tau\otimes\sgn),
\end{array}
\right.\\\\
\textup{ otherwise, } &\quad& \left\{
\begin{array}{l}
\ker D_\tau^+ \cong \fM_{1,c}(\tau)\\
\ker D_\tau^- = 0.
\end{array}  
\right.
\end{array}
\]
\end{corollary}

\begin{proof}
We just need to show that the kernel of $D^{\pm}_\tau$ is contained in a subspace of $\fX_{t,c}(\widetilde\tau)$ as specified since the other inclusions follow from Corollary \ref{c:lowdegree} and Proposition \ref{p:kernelD}. Suppose that $\xi$ is in the kernel of $D_\tau = D_\tau^++D_\tau^-$. There are two options: either $\xi\in F_0(\fX_{1,c}(\widetilde\tau))$ or not. If yes, then $\xi$ is  in $F_0(\fX_{1,c}(\widetilde\tau))\cong \fM_{1,c}(\tau\otimes {\bigwedge}^{\!0}\fh)\oplus \fM_{1,c}(\tau\otimes {\bigwedge}^{\!1}\fh)$. From Corollary \ref{c:lowdegree}, $\fM_{1,c}(\tau\otimes {\bigwedge}^{\!0}\fh)\subseteq\ker D_\tau^+$, for all values of $c$. So suppose that $\xi$ is in $\fM_{1,c}(\tau\otimes {\bigwedge}^{\!1}\fh)$. But note that $D_\tau = D_\tau^-$ in that space and 
\[D^-_\tau(p\otimes 1 \otimes z_\tau\otimes y) = 2\partial(p)\otimes 1 \otimes z_\tau\otimes 1 +2(1+2c\tau(s))p\otimes y \otimes z_\tau\otimes 1.\]
This can only be zero if $p\in\bbc$ and $1 + 2c\tau(s) = 0$. So $\xi = \xi_{0,0}$ and $Q(\chad)\cdot(\xi_{0,0}) \cong L_{1,-\tau(s)\frac{1}{2}}(\tau\otimes\sgn)$. These are the possibilities for $\xi\in F_0(\fX_{1,c}(\widetilde\tau))$.

For the other case, upon replacing $\xi$ by $Q(y^m)\xi$, for a sufficiently big $m$, we can assume that $\xi \in \bbc\otimes \bbc[y]_{\leq n}\otimes \bbc_\tau\otimes \bigwedge \fh$, for an $n>0$. From Corollary \ref{c:lowdegree}, we have $D_\tau$ is never zero on $\bbc\otimes \bbc[y]_d\otimes \bbc_{\tau}\otimes {\bigwedge}^{\!0}\fh$, if $d>0$, so it must be that $\xi \in \bbc\otimes \bbc[y]_{\leq n}\otimes \bbc_{\tau}\otimes {\bigwedge}^{\!1}\fh$, and $D_\tau$ acts as $D_\tau^-$ on $\xi$.  In the present case case, we can thus write
\[\xi = 1\otimes q \otimes z_\tau\otimes y,\]
for a polynomial $q = a_0 + a_1 y+\cdots + a_n y^n$ and $0\neq z_\tau\in\bbc_\tau$. Thus, $\xi = \sum_{d} a_d\xi_{0,d}$. By linear independence, $D_\tau(\xi) = 0$ implies $a_dD_\tau(\xi_{0,d}) = 0$ for all $d$. From Corollary \ref{c:computation}, $D(\xi_{0,d}) = 0$ if and only if $d=2k$ is even and $2k+1 + 2c\tau(s) = 0$. We thus showed that there exists a $k>0$ (we are assuming that $\xi\notin F_0(\fX_{1,c}(\widetilde\tau))$) such that $2c\tau(s) + (2k+1) = 0$ and our initial $\xi$ is in $\big(Q(\cha)\cdot\xi_{0,2k}\big)\cong L_{1,c}(\tau\otimes\sgn)$. 
\end{proof}

\begin{remark}
The second case in the Corollary \ref{c:kerZ2} includes the situation in which $c$ is regular and also if $c$ is not regular, but $\tau(s)+c/|c|\neq 0$.
\end{remark}
We now look at the image and cokernel of $D_\tau$. We shall treat the operators $D_\tau^{\pm}$ separately. 

\begin{lemma}\label{l:Z2genImage} 
The operator $D_\tau^+:\fX_{t,c}(\tau\otimes{\bigwedge}^{\!0}\fh)\to\fX_{t,c}(\tau\otimes{\bigwedge}^{\!1}\fh)$ is surjective.
\end{lemma}
\begin{proof}
Clear, since $D_\tau^+ = 1\otimes\partial\otimes 1\otimes \fs(y)$, and both $\partial\in\End(\bbc[y])$ and $\fs(y):{\bigwedge}^{\!0}\fh\to {\bigwedge}^{\!1}\fh$ are surjective.
\end{proof}

\begin{lemma}\label{l:image-c-regular}
Suppose there is no $k\in\bbz_{\geq 0}$ such that $2k+1 + 2\tau(s)c = 0$. Then, the image of $D_\tau^-$ contains the subspace
\[\bigoplus_{m-d<0}\bbc[x]_m\otimes\bbc[y]_d \otimes \bbc_\tau \otimes {\bigwedge}^{\!0}\fh\]
\end{lemma}

\begin{proof}
From the assumption on $c$, we have, from Lemma \ref{l:computation}, that $D_\tau^-(\xi_{m,d}) = m\psi_{m-1,d}+b\psi_{m,d+1}$ with $b\in\bbc$ nonzero. It follows that the subspace $\oplus_{d>0}\fX_{t,c}(\widetilde\tau)_{0,d}^0$ is in the image, as $D_\tau^-(\xi_{0,d}) = b\psi_{0,d+1}$ for a nonzero $b\in\bbc$. Assume, by induction on $n$, that the space
$\oplus_{d>0}\fX_{t,c}(\widetilde\tau)_{n,n+d}^0$ is in the image, for $n>0$. As
\[D_\tau^-(\xi_{n+1,n+1+d}) = a\psi_{n,n+1+d}+b\psi_{n+1,n+1+d+1}\]
with $a,b$ nonzero, the induction hypothesis implies that $\fX_{t,c}(\widetilde\tau)_{n+1,n+1+d}^0$ is in the image for all $d>0$, and hence the space $\oplus_{n\geq0}\oplus_{d>0}\fX_{t,c}(\widetilde\tau)_{n,n+d}^0$ is in the image, finishing the proof.
\end{proof}

\begin{proposition}\label{p:coker-c-regular}
Suppose there is no $k\in\bbz_{\geq 0}$ such that $2k+1 + 2\tau(s)c = 0$. Then, the cokernel of $D_\tau^-$ is linearly isomorphic to $\bbc[x]\otimes \bbc \otimes \bbc_\tau\otimes {\bigwedge}^{\!0}\fh$. It follows that $\coker D_\tau^{-}\cong \fM_{1,c}(\tau),$ as $\cha$-modules.
\end{proposition}

\begin{proof}
For each $n\in \bbz$, consider the diagonal spaces $P_n=\oplus_{m-d = n}\fX_{t,c}(\widetilde\tau)_{m,d}^0$, with $m,d\in\bbz_{\geq 0}$. It is then clear that $\fX_{t,c}(\tau\otimes{\bigwedge}^{\!0}\fh) = \oplus_{n\in\bbz} P_n$. From the previous Lemma, we get that $P_n = 0$ in $\coker(D_\tau^-)$, if $n<0$. We now claim that, modulo the image of $D_\tau^-$, each subspace $P_n$ is one-dimensional and represented by $\fX_{t,c}(\widetilde\tau)_{n,0,0} = \bbc[x]_n\otimes\bbc\otimes\bbc_\tau\otimes\bbc$, for $n\geq 0$. The claim will finish the proof. But for that, note from Lemma \ref{l:computation} that
\[
\tfrac{1}{2}D_\tau^-(\xi_{n+d,d-1}) = (n+d)\psi_{n+d-1,d-1}+b\psi_{n+d,d},
\]
for nonzero constant $b\in\bbc$, because of the assumption on $c$. By induction, we have that, for all $d>0$, $\psi_{n,0} = \lambda_d\psi_{n+d,d}$ in $\coker(D_\tau^-)$ for a nonzero $\lambda_d\in\bbc$.
\end{proof}

\begin{lemma}\label{l:image-c-sing}
Suppose that the parameter $c$ satisfies $2c =-\tau(s)(2k+1)$, where $s$ is the nontrivial element of $\bbz_2$. Then, the image of $D_\tau^-$ contains the subspace
\[\left(\bigoplus_{m-d<2k+1}\bbc[x]_m\otimes\bbc[y]_d \otimes \bbc_\tau \otimes {\bigwedge}^{\!0}\fh\right)\oplus\Big(\bbc[x]\otimes\bbc[y]_{\leq 2k}\otimes\tau\otimes{\bigwedge}^{\!0}\fh\Big).\]
\end{lemma}

\begin{proof}
First note that the choice of $c$ forces $\tfrac{1}{2}D_\tau^-(\xi_{m,2k}) = m\psi_{m-1,2k}$. This, coupled with the fact that, if $d\neq 2k$, then
\begin{equation}\label{e:important}
\tfrac{1}{2}D_\tau^-(\xi_{m,d}) = m\psi_{m-1,d}+b\psi_{m,d+1}
\end{equation}  
with $b$ nonzero, implies that $\bbc[x]\otimes\bbc[y]_{\leq 2k}\otimes\tau\otimes{\bigwedge}^{\!0}\fh$ is in the image. 

The proof that the other part of the statement is in the image is completely analogous to the proof of Lemma \ref{l:image-c-regular}, by shifting $d\to d+(2k+1)$. Note that the equation (\ref{e:important}) has $b\in\bbc$ nonzero for all $d\geq 2k+1$.
\end{proof}

\begin{proposition}\label{p:coker-c-singular}
Suppose that the parameter $c$ satisfies $2c =-\tau(s)(2k+1)$, where $s$ is the nontrivial element of $\bbz_2$. Then, the cokernel of $D_\tau^-$ is linearly isomorphic to $\bbc[x]\otimes \bbc[y]_{2k+1} \otimes \bbc_\tau\otimes {\bigwedge}^{\!0}\fh$. It follows that $\coker D_\tau^{-}\cong \fM_{1,c}(\tau\otimes\sgn),$ as $\cha$-modules.
\end{proposition}

\begin{proof}
The proof is similar to the one of Proposition \ref{p:coker-c-regular}. We define $Q_n = \oplus_{m-d = n-(2k+1)}\fX_{t,c}(\widetilde\tau)_{m,d}^0$ (with $m,d\in\bbz_{\geq 0}$ and $n\in\bbz$) and note that
\[\fX_{t,c}\left(\tau\otimes{\bigwedge}^{\!0}\fh\right) = \Big(\bbc[x]\otimes\bbc[y]_{\leq 2k}\otimes\tau\otimes{\bigwedge}^{\!0}\fh\Big)\oplus\Big(\oplus_{n\in \bbz}Q_n\Big).\]
Then, modulo the image of $D_\tau^-$, each diagonal $Q_n$ with $n<0$ becomes $0$ (Lemma \ref{l:image-c-sing}) while with $n\geq0$ it is one dimensional and represented by $\fX_{t,c}(\widetilde\tau)_{n,2k+1}^0$. It is important to remark that the equation
\[
D_\tau^-(\xi_{n+d,2k+d}) = a\psi_{n+(d-1),2k+1+(d-1)}+b\psi_{n+d,2k+1+d},
\]
also has nonzero constants $a,b\in\bbc$, if $d>0$,  from which $\psi_{n,2k+1} = \lambda_d\psi_{n+d,2k+1+d}$ in $\coker(D_\tau^-)$ for a $\lambda_d\neq 0$, for all $d>0$.
\end{proof}

We summarise this example with the  diagrams for
\[
\begin{tikzcd}
0\xar{r}&\ker D_\tau^{\varepsilon}\xar{r} &\fX_{1,c}(\widetilde\tau)^{\varepsilon} \xar{r}{D_\tau^{\varepsilon}} & \fX_{1,c}(\widetilde\tau)^{-\varepsilon} \xar{r} &\coker D_\tau^{\varepsilon}\xar{r}&0.
\end{tikzcd}
\]
If $c=-\tau(s)(2k+1)/2$ we have
\[
\begin{tikzcd}
0\xar{r}&\fM_{1,c}(\tau)\xar{r} &\fX_{1,c}(\widetilde\tau)^{+}\xar{r}{D^+_\tau}& \fX_{1,c}(\widetilde\tau)^{-} \xar{r} &0\xar{r}&0\\
0\xar{r}&L_{1,c}(\tau\otimes\sgn)\xar{r} &\fX_{1,c}(\widetilde\tau)^{-} \xar{r}{D^-_\tau} & \fX_{1,c}(\widetilde\tau)^{+} \xar{r} &\fM_{1,c}(\tau\otimes\sgn)\xar{r}&0,
\end{tikzcd}
\]
while otherwise, we have
\[
\begin{tikzcd}
0\xar{r}&\fM_{1,c}(\tau)\xar{r} &\fX_{1,c}(\widetilde\tau)^{+} \xar{r}{D^+_\tau} & \fX_{1,c}(\widetilde\tau)^{-} \xar{r} &0\xar{r}&0\\
0\xar{r}&0\xar{r} & \fX_{1,c}(\widetilde\tau)^{-}\xar{r}{D^-_\tau} & \fX_{1,c}(\widetilde\tau)^{+} \xar{r} &\fM_{1,c}(\tau)\xar{r}&0.
\end{tikzcd}
\]
In both cases, we have $I(\tau) = \fM_{t,c}(\tau) = -I(\tau)^-$, in the Grothendieck group of $\Co$.


\begin{thebibliography}{99}

\bibitem[AS]{AS} M.~Atiyah, W.~Schmid,
\emph{A geometric construction of the discrete series for semisimple Lie groups}, Invent. Math. {\bf 42} (1977), 1--62.

\bibitem[BP]{BP} M.~Balagovic, A.~Puranik,
{\it Irreducible representations of the rational {C}herednik algebra associated to the Coxeter group $H_3$},
J. Algebra {\bf 405} (2014), 259--290.

\bibitem[BCT]{BCT}
D.~Barbasch, D.~Ciubotaru, P.~Trapa,
\emph{Dirac cohomology for graded affine Hecke algebras} Acta Math. {\bf 209} (2012), no. 2, 197--227.

\bibitem[BT]{BT} G.~Bellamy, U.~Thiel,
\emph{Cuspidal Calogero-Moser and Lusztig families for Coxeter groups}, J. Algebra {\bf 462} (2016), 197--252.

\bibitem[BEG]{BEG}  Y.~Berest, P.~Etingof, V.~Ginzburg, 
\emph{Finite-dimensional representations of rational Cherednik algebras}, Int. Math. Res. Not. 2003, no. {\bf 19}, 1053--1088.

\bibitem[Ch]{Ch}  I.~Cherednik,  
\emph{Double affine {H}ecke algebras. {K}nizhnik-{Z}amolodchikov equations and {M}acdonald operators,}, IMRN (Duke Math. J.) 9 (1992), 171--180.

\bibitem[Cha]{Cha} K.Y.~Chan,
\emph{Duality for Ext-groups and extensions of discrete series for graded affine Hecke algebras}, Adv. Math. {\bf 294} (2016), 410--453.

\bibitem[Ci1]{Ci} D.~Ciubotaru,
\emph{Dirac cohomology for symplectic reflection algebras}, Selecta Math. (N.S.) {\bf 22} (2016), no. 1, 111--144.

\bibitem[Ci2]{Ci2} D.~Ciubotaru,
\emph{One-$W$-type modules for rational Cherednik algebras and cuspidal two-sided cells}, to appear in Bull. Acad. Math. Sinica.

\bibitem[COT]{COT}
D.~Ciubotaru, P.~Trapa, E.~Opdam,
\emph{Analytic and algebraic induction for graded affine Hecke algebras},
J. Inst. Math. Jussieu {\bf 13} (2014), no. 3, 447--486. 

\bibitem[CT]{CT}
D.~Ciubotaru, P.~Trapa,
\emph{Characters of Springer representations on elliptic conjugacy classes}, Duke Math. J. {\bf 162} (2013), no. 2, 201--223.

\bibitem[DJO]{DJO} C.~Dunkl, M.~ De Jeu, E.~Opdam,
\emph{Singular polynomials for finite reflection groups}, Trans. Amer. Math. Soc. {\bf 346} (1994), 237--256.

\bibitem[DO]{DO} C. Dunkl, E.~Opdam,
\emph{Dunkl operators for complex reflection groups}, Proc. London Math. Soc. (3) {\bf 86} (2003), 70--108.

\bibitem[EG]{EG} P.~Etingof, V.~Ginzburg, 
{\it Symplectic reflection algebras, Calogero-Moser space and deformed Harish-Chandra homomorphism}, Invent. Math. {\bf 147} (2002), no. 2, 243--348. 

\bibitem[EM]{EM} P.~Etingof, X.~Ma, 
{\it Lecture notes on Cherednik algebras}, \texttt{arXiv}.

\bibitem[EOS]{EOS} E.~Emsiz, E.~Opdam, J.~Stockman,
\emph{Trigonometric Cherednik algebra at critical level and quantum many-body problems}, Selecta Math. (N.S.) {\bf 14} (2009), no. 3-4, 571--605. 

\bibitem[Et]{Et} P.~Etingof,
{\it Supports of irreducible spherical representations of rational {C}herednik algebras of finite {C}oxeter groups},
Adv. math. 229 (2012) 2042--2054.

\bibitem[GGOR]{GGOR} V.~Ginzburg, N.~Guay, E.~Opdam, R.~Rouquier,
\emph{On the category $\Co$ for rational Cherednik algebras}, Invent. Math. {\bf 154} (2003), no. 3, 617--651. 

\bibitem[Go]{Go} I.~Gordon,
\emph{Baby Verma modules for rational Cherednik algebras}, Bull. London Math. Soc. {\bf 35} (2003), no. 3, 321--336.

\bibitem[Gu]{Gu} E.~Gutkin,
\emph{Operator calculi associated with reflection groups}, 
Duke Math. J. {\bf 55} (1987), no. 1, 1--18. 

\bibitem[HO]{HO} G.~Heckman, E.~Opdam,
\emph{Yang's system of particles and Hecke algebras}, Ann. of Math. (2) {\bf 145} (1997), no. 1, 139--173.

\bibitem[HP]{HP} J.S.~Huang, P.~Pand\v zi\' c,
\emph{Dirac operators in representation theory}, Mathematics: Theory \& Applications. Birkh\" auser Boston Inc., Boston, MA, 2006.

\bibitem[OS]{OS} E.~Opdam, M.~Solleveld,
\emph{Homological algebra for affine {H}ecke algebras}, Adv. Math. 220(5) (2009), 1549--1601.

\bibitem[Pa]{Pa} R.~Parthasarathy,
\emph{Dirac operator and the discrete series}, Ann. of Math. (2) {\bf 96} (1972), 1--30.

\bibitem[Re]{Re} M.~Reeder,
\emph{{E}uler-{P}oincar\'e pairings and elliptic representations of {W}eyl groups and $p$-adic groups}, Compositio Math. 129(2) (2001), 149--181.

\bibitem[Vo]{Vo} D.A.~Vogan,~Jr,
\emph{Lectures on Dirac operators I-III}, talks in the M.I.T. Lie groups seminar, 1997.

\bibitem[We]{We} C.A.~Weibel,
\emph{The {K}-book: an introduction to algebraic {K}-theory}, Graduate Studies in Mathematics, Volume 145. American Mathematical Society, Providence, RI, 2013.
\end{thebibliography}
\end{document}